\documentclass[10pt, reqno, oneside, english]{smfart}

\usepackage{smfthm}

\usepackage[height=22cm, bottom=3cm]{geometry}
\usepackage{amsmath, amsthm, amssymb,amsfonts, amscd, mathrsfs}
\usepackage[utf8]{inputenc}
\usepackage{url}
\usepackage{braket}
\usepackage{enumerate}
\usepackage{xcolor}
\usepackage[colorlinks=true, linktoc=page, citecolor=blue, linkcolor=blue, urlcolor=blue]{hyperref}

\usepackage{setspace}
\numberwithin{equation}{section}

\newcommand{\norme}[1]{\left\Vert #1\right\Vert}

\title[Invariant Gaussian fields and their nodal volume]{Invariant Gaussian fields on homogeneous spaces: \\ explicit constructions and mean nodal volume}
\author{Alexandre Afgoustidis}
\address{CNRS \& Institut Élie Cartan de Lorraine, Nancy \& Metz, France}
\email{alexandre.afgoustidis@math.cnrs.fr}

\usepackage{graphicx}

\begin{document}

\frontmatter
\begin{abstract} We review and study some of the properties of smooth Gaussian random fields defined on a homogeneous space, under the assumption that the probability distribution is invariant under the isometry group of the space. We first give an exposition, building on early results of Yaglom, of the way in which representation theory and the associated special functions make it possible to give completely explicit descriptions of these fields in many cases of interest. We then turn to the expected size of the zero-set: extending two-dimensional results from Optics and Neuroscience, we show that every invariant field comes with a natural unit of volume (defined in terms of the geometrical redundancies in the field) with respect to which the average size of the zero-set is given by a universal constant depending only on the dimension of the source and target spaces, and not on the precise symmetry exhibited by the field.
\end{abstract}

\mainmatter
\maketitle


\section{Introduction}\label{sec:intro}

\subsection{Nodal sets of real-valued random fields} \label{subsec:intro_reel}
Suppose $X$ is a Riemannian manifold and $\psi: X \to \mathbb{R}$ is a random function. {It is a classical problem to try to understand the geometry of the \emph{nodal set} $\psi^{-1}(0)$ of the samples of $\psi$, especially when each sample of $\psi$ is an eigenfunction of the Laplace-Beltrami operator on~$X$.

It is not an easy problem. At the present time, very little can be said unless $\psi$ is a Gaussian random field or can be obtained from one by a simple modification. Even when $X$ is the two-sphere and $\psi$ is a real-valued Gaussian field with values in a (finite-dimensional) eigenspace of the Laplacian, it is not easy to study such an apparently simple quantity as the mean number of connected components of the nodal set. When $X$ is a compact manifold and $\psi$ is a real-valued Gaussian field with values in a (finite-dimensional) eigenspace of the Laplacian, much effort is being directed at understanding the Betti numbers of the nodal set of $\psi$ and the way they depend on the eigenvalue; see \cite{Anantharaman}.

If one turns from the topology to the \emph{size} of the nodal set, then more can be said. For \emph{deterministic} functions, Yau conjectured that in a region of $X$ of fixed volume, the nodal volume of \emph{any} eigenfunction on the Laplacian on $X$ is bounded above and below by constant multiples of the square-root of the eigenvalue \cite{Yau1, Yau2}. The conjecture has led to spectacular theorems in the analytic case \cite{DonnellyFefferman}, and recently in the smooth case \cite{Logunov}.

For \emph{random} functions coming from Gaussian fields, the \emph{mean nodal volume} has been made explicit for several fields with values in an eigenspace of the Laplacian on some particularly symmetric compact manifolds: see \cite{Berard} for spheres and other compact rank-one symmetric spaces, \cite{RudnickWigman} for flat tori. In harmony with Yau's conjecture, the mean volume always turns out to be the product of the square-root of the eigenvalue with a constant that depends only on the dimension of the manifold. The variance and distribution of the nodal volume are currently being subjected, for particular spaces $X$, to  intense scutiny: see \cite{KKW, PeccaReel, Pecca, Cammarota, RossiWigman} for flat tori, \cite{Rossi} for spheres. Symmetry is an important ingredient, via Fourier analysis, for these works.

This article tries to illustrate the idea that if $X$ has sufficiently many symmetries to be equipped with an isometric and transitive action of a Lie group $G$, then the symmetry greatly simplifies the study of $G$-invariant smooth random fields (we do not assume $X$ to be compact). 

The paper has two themes. The first is the interplay between invariant random fields on $X$ and the representation theory of $G$, via spherical functions. We shall give an overview of the relationship, and provide explicit constructions on certain important examples of spaces $X$. The discussion will be mostly expository. 

The second theme is the mean nodal volume of invariant Gaussian fields: we shall give a precise estimate of for the average size of the nodal set for general $X$. For instance, Theorem \ref{Main} below, when specialized to real-valued fields with samples in an eigenspace of the Laplacian and combined with Proposition \ref{Monoch}, will yield the following result.

\begin{theo}[special case of Theorem \ref{Main}] Let $X$ be a Riemannian manifold equipped with a transitive metric-preserving action of a Lie group $G$. Let $\psi$ be a smooth real-valued Gaussian random field on $X$ such that $\psi(x)$ follows, for every $x \in X$, a standard normal distribution. Assume that the probability distribution of $\psi$ is $G$-invariant and that there exists a real number~$E$ such that almost all samples of $\psi$ lie in the eigenspace of the Laplace-Beltrami operator $\Delta_X$ for the eigenvalue $E$.

The nodal set $\psi^{-1}(\{0\})$  is generically a hypersurface; for any Borel region $B$ of $X$, write $\mathcal{V}_A$ for the real-valued random variable recording the volume of $\psi^{-1}(\{0\}) \cap B$. Then 
$$ \mathbb{E}[\mathcal{V}_B] = c \cdot \sqrt{E}\cdot \text{Vol}(B)$$
where $c$ is a positive number that depends only on the dimension of $X$.
\end{theo}

\subsection{Zeroes of complex-valued fields on the plane and an observation from Neuroscience} \label{subsec:intro_pinwheels} The work that led to this paper did not begin with \emph{real}-valued fields in mind, however.  Let us describe some simple facts about some complex-valued Gaussian fields on the plane that have proved useful in Neuroscience, in relation with a striking fact observed a few years ago \cite{Kaschube} in the \emph{primary visual cortex} of mammals. 

Part of the neurons' specialization in that area can be described, for each individual in many mammalian species, with the help of a continuous \emph{complex-valued} map defined on the cortical surface. This map is called the \emph{orientation map}, and its zeroes are of particular biological significance (see \cite{Hubel}). If we assimilate the cortical surface, in the central region of the primary visual area, with a Euclidean plane, then the orientation map can be assimilated with a function $\mathbf{z}_{{\mathrm{exp}}}: \mathbb{R}^2 \to \mathbb{C}$, {which can be measured experimentally}. The traditional wisdom is that {the observed map} $\mathbf{z}_{{\mathrm{exp}}}$ is very roughly a combination of plane waves with various wavevector directions and various phases, but a common wavelength $\Lambda_{{\mathrm{exp}}}$ (which depends on the species and the individual). It has been observed that the average density of the zero-set of $\mathbf{z}_{{\mathrm{exp}}}$ is strikingly similar across individuals and species:

\begin{enonce*}{Experimental fact}[\cite{Kaschube}] Let $\Lambda_{{\mathrm{exp}}}$ be the characteristic wavelength of a cortical orientation map $\mathbf{z}_{{\mathrm{exp}}}$, and $\mathcal{N}_{{\mathrm{exp}}}$ stand for the average number of zeroes of $\mathbf{z}_{{\mathrm{exp}}}$ in a region with area $\Lambda_{{\mathrm{exp}}}^2$ of the map. The number $\mathcal{N}_{{\mathrm{exp}}}$ has been measured in individual cortical maps coming from quite different species. The value of $\mathcal{N}_{{\mathrm{exp}}}$ in each individual cortex is close to 3.14. \end{enonce*} %

Useful models for the early stage in orientation map development \cite{WolfGeisel, WG2003} treat the experimental map $\mathbf{z}_{{\mathrm{exp}}}$ in a given individual as a single realization of a random field. In fact, one usually treats $\mathbf{z}_{{\mathrm{exp}}}$ as a realization of Gaussian random field $\psi: \mathbb{R}^2 \to \mathbb{C}$ whose samples are in an eigenspace of the Laplacian, and one crucially adds the further assumption, meant to reflect the initial homogeneity of the biological tissue, that $\psi$ is invariant under all translations and rotations of $\mathbb{R}^2$  (see \cite{AAVariance} for a discussion). In this context, Wolf and Geisel discovered the following mathematical fact \cite{WolfGeisel}, almost simultaneously exhibited in Optics \cite{BerryDennis}.

 \begin{theo}[\cite{WolfGeisel, BerryDennis}] \label{th:wolf_geisel} Suppose $\psi$ is a stationary and isotropic Gaussian field on $\mathbb{R}^2$ with values in $\mathbb{C}$, assume that the complex-valued variable $\psi(x)$ follows a standard Gaussian distribution for all $x \in \mathbb{R}^2$, and assume that there exists $\Lambda>0$ such that almost all samples of $\psi$ satisfy $\Delta \psi = \left( \frac{2\pi}{\Lambda} \right)^2\psi$. The expectation for the number of zeroes of $\psi$ in a region with area $\Lambda^2$ is~$\pi$.\end{theo}

It was an attempt to assess the exact role of symmetry arguments in this result, and to adapt the models to non-Euclidean geometries, that led to the work reported in this paper; the present article is a mathematical outgrowth of the former articles \cite{AAPinwheels} and \cite{AAVariance}. We accordingly will study Gaussian fields on Riemannian homogeneous spaces, but in contrast to the situation of \S \ref{subsec:intro_reel}, we shall allow for the target space to be an arbitrary finite-dimensional vector space. Theorem \ref{Main}, when specialized to the case where the source and target spaces have the same dimension and the field has samples in an eigenspace of the Laplace-Beltrami operator, will yield the following generalization of Wolf and Geisel's result.

\begin{theo}[special case of Theorem \ref{Main}] Let $X$ be a Riemannian manifold equipped with a transitive metric-preserving action of a Lie group $G$. Let $\psi$ be a smooth Gaussian random field on $X$, with values in a $\dim(X)$-dimensional Euclidean space $V$, such that the random vector $\psi(x)$ of $V$ follows, for every $x \in X$, a standard normal distribution. Assume that the probability distribution of $\psi$ is $G$-invariant and that there exists a positive number $\Lambda$ such that almost all samples of $\psi$ lie in the eigenspace of the Laplace-Beltrami operator $\Delta_X$ for the eigenvalue $E=\left(\frac{2\pi}{\Lambda}\right)^2$. Assume further that the coordinates of $\psi$, in some orthonormal basis for $V$, are independent as random fields.

The nodal set $\psi^{-1}(\{0\})$  is generically a discrete set; write $\mathcal{N}$ for the real-valued random variable recording the number of zeroes of $\psi$ in a region with volume $\Lambda^{\dim(X)}$. Then 
\begin{equation} \label{wg} \mathbb{E}[\mathcal{N}] = {(\dim X )! \left( \frac{\pi}{2} \right)^{(\dim X)/2}}.\end{equation}
\end{theo}

{Specializing to the two-dimensional plane, and its Euclidean motion group,  of course gives the Wolf--Geisel density result (Theorem \ref{th:wolf_geisel}) $-$ where the expected nodal density $\mathbb{E}[\mathcal{N}]$ so strikingly coincidates with the experimentally measured density $\mathcal{N}_{{\mathrm{exp}}}$ of orientation maps. }

\subsection{On the construction of invariant Gaussian fields}\label{subsec:intro_yaglom} The results described in \S \ref{subsec:intro_reel} and \S \ref{subsec:intro_pinwheels} are answers to special cases of the following question, which we will address in \S \ref{sec:spacing} and~\S\ref{sec:densite}.

\begin{enonce*}{Problem A} Let $X$ be a smooth Riemannian manifold equipped with a transitive action of a Lie group $G$ by isometries, and let $V$ be a Euclidean space. Let $\Phi: X \to V$ be a Gaussian random field whose samples are a.s. smooth and whose probability distribution is $G$-invariant. What is the average size of the zero-set of $\Phi$?\end{enonce*}

Our strategy relies on some of the ideas encountered in \cite{AAPinwheels} in relation with Neuroscience: we shall analyze the geometrical redundancies in the field $\Phi$ to define, when $\Phi$ is real-valued a ``characteristic wavelength'' $\Lambda(\Phi)$, and when $\Phi$ is vector-valued, a ``characteristic volume unit'' $\mathcal{V}(\Phi)$. We will then show that the average size of the nodal set can be expressed in terms of these quantities in a manner close to \eqref{wg}; the argument will turn out to involve only general facts about Gaussian fields and Riemannian geometry (together with a powerful version of the Kac-Rice formula for random fields). We will not need to have information on the precise structure of the field or on the kind of symmetry encoded by the group~$G$; but symmetry will be crucial in simplifying the analysis. 

A consequence that may lead to some psychological discomfort is that Problem A can be studied even in the absence of concrete knowledge of what the fields $\Phi$ are. In many interesting cases, however, it is possible to have a very precise idea of what the invariant fields $\Phi$ on $X$ look like, and therefore to answer the following question.

\begin{enonce*}{Problem B} Let $X$ be a Riemannian $G$-homogeneous space as above, and let $V$ be a Euclidean space Describe as explicitly as possible the $G$-invariant smooth Gaussian random fields $\Phi: X \to V$ whose probability distribution is $G$-invariant. \end{enonce*}

This problem is of course of independent interest $-$ it dates back to Kolmogorov \cite{Kolmogorov}. Very general and powerful information was given in 1960 by Yaglom \cite{Yaglom}, who established a deep connection between Problem B and the representation theory of $G$. Yaglom also recognized that this problem is only tractable for special classes of homogeneous spaces $X$; we shall in fact only consider Problem B for  $V = \mathbb{C}$, and limit ourselves to examples of spaces $X$ for which extensive information about the representation theory of $G$ (and its application to harmonic analysis on $X$) is available.

Aside from providing psychological help for our study of Problem A, it is perhaps not absurd to discuss in some detail a few aspects of Problem B, more than fifty years after Yaglom: 
\begin{itemize}
\item[$\bullet$] Some of the aspects of representation theory that are related to Problem B have been made quite explicit over the past decades; in several important cases, the available results are so concrete that it is easy to describe all smooth $G$-invariant Gaussian fields $\Phi: X \to \mathbb{C}$ in a manner practical enough to make it possible (in principle) to simulate every invariant field on a computer.
\item[$\bullet$]  Interest for smooth Gaussian random functions with symmetry properties has risen recently in relation with several applications: let us mention Neuroscience \cite{WolfGeisel}, Optics \cite{BerryDennis} and Sismology \cite{Zerva} in relation with waves diffracted in unpredictable directions, Cosmology \cite{MarinucciPeccatti,Malyarenko} in relation with the study of the Cosmic Microwave Background, and Image Processing \cite{GalerneMorel, Wei} in relation with textures.
\end{itemize}
Considering these two points, it seems that it may be welcome to set down constructions of invariant Gaussian fields on homogeneous cases in several examples of general interest, and to do so in as explicit a manner as possible $-$ even though the general theory is entirely due to Yaglom.

\subsection{Outline of the paper} \label{subsec:outline}

Problems A and B are both intimately related with the fact that one can read off the probability distribution of a real- or complex-valued Gaussian field (and therefore also, in principle, all statistical properties of the field) from its covariance function. We will review the necessary facts in \S \ref{subsec:gauss}. 

We will then consider Problem B. The relationship between invariant Gaussian fields and group representations rests on the observation, due to Yaglom and recalled in \S \ref{subsec:repthy}, that the class of covariance functions of invariant smooth Gaussian fields has an immediate interpretation in terms of matrix elements of unitary representations.

Just as an arbitrary unitary representation can be, in favorable circumstances, expressed in terms of irreductible representations, one may expect invariant Gaussian fields to decompose into a ``sum'' of elementary ones. But the ``decomposition'' theory of unitary representations is tractable only when certain (mild) restrictions on the group $G$ are imposed; accordingly, a good ``decomposition theory'' for random fields is available only for certain classes of spaces $X$.  In \S \ref{subsec:gelfand} and \S \ref{subsec:diffop}, we focus on the particular class of \emph{commutative spaces}. On these special homogeneous spaces, every invariant field can be decomposed as a continuous sum of ``elementary'' fields (called \emph{monochromatic} in our setting), which are related\footnote{The conditions on $X$ for the existence of a ``good decomposition theory'' for random fields are somewhat stronger than the conditions on $G$ for the existence of a ``good decomposition theory'' for representations: see Definition \ref{commut}.} to irreducible unitary representations of $G$. We recall Yaglom's results in \S \ref{subsec:gelfand}, then point out in \S \ref{subsec:diffop} that the spectral theory of $G$-invariant differential operators can provide concrete information about these elementary fields. As an illustration of the fact  that the existence and tractability of smooth Gaussian random fields on a homogeneous space $X$ imposes nontrivial conditions on $X$, we study in \S \ref{subsec:affine} a class of simple examples which show that on a homogeneous space that is not commutative, the theory of smooth invariant random fields can break down completely. 

In \S \ref{sec:exemples}, we use the above results to describe classifications of the invariant fields on a general class of flat commutative spaces, on all positively-curved commutative spaces, and on a special (but very useful) class of negatively-curved spaces.  

It should be very clear that \S \ref{sec:generalites} and \S \ref{sec:exemples} are to a large extent a synthesis and an exposition of well-known theorems and methods, most of which are a half-century old. We merely intend to point out that Yaglom's general results can now, for several classes of interest, be given an extremely concrete form. Our aim in \S \ref{sec:exemples} is to give his results a practical enough shape to allow for computer simulation of all $G$-invariant fields, conditional on the evaluation of some concrete invariants that appear in the representation theory of $G$. {Everything relies on well-known facts: in particular, the discussion in \S \ref{subsec:gelfand} is closely based on Yaglom and that of \S \ref{subsec:diffop} has become standard lore in invariant harmonic analysis. However, it seems the counterexample of \S \ref{subsec:affine} and the fully explicit descriptions on the  classes of examples in \S \ref{sec:exemples} may be new in this generality. (For the negatively-curved case of \S \ref{subsec:hyperbolique}, see also the recent preprint \cite[Section 3.6]{AbertBergeronLeMasson}.) }

This material will hopefully furnish enough background for our analysis of the average size of the nodal set, Problem A, to which we turn in \S \ref{sec:spacing} and \S \ref{sec:densite}. Again, our results are simple and straightforward consequences of deep and general Kac-Rice formulae \cite{AdlerTaylor, AzaisWschebor}, making the whole paper somewhat expository.

We will work in the general context of Riemannian homogeneous spaces (without the commutativity hypothesis) and will not need to call in the link with representation theory: in fact, the only results technically necessary for our proofs in \S \ref{sec:spacing}-\ref{sec:densite} are those of  \S \ref{subsec:gauss} and \S \ref{subsec:repthy}.

In \S \ref{sec:spacing}, we shall attach to any real-valued invariant field $\Phi$ a characteristic length $\Lambda(\Phi)$: if one moves along a geodesic in $X$, this is the average length separating two points where the field takes the same value. A simple application of the one-dimensional Kac-Rice formula will show how this can be evaluated from the spectral theory of the Laplace-Beltrami operator.

In \S \ref{sec:densite}, we state and prove our result on the average nodal volume in a given region of space. The proof is, not surprisingly in view of related studies, a rather standard application of one of the recent  (and powerful) versions  of the Kac-Rice formula for random fields $-$ the one we shall use is due to Aza\"is and Wschebor \cite[Chapter 6]{AzaisWschebor}. 

We emphasize that {although the characteristic length and volume units introduced in \S \ref{sec:spacing}-\ref{sec:densite} do seem new}, the methods used to obtain the results are quite standard and that no important technical obstacle awaits us in the proofs. In fact, after an earlier version of this paper was written, it became apparent that the result of \S \ref{sec:densite} is very close to being an extremely special case of Adler and Taylor's deep theorems on Lipschitz-Killing curvatures, especially \cite[Theorem 15.9.4]{AdlerTaylor}.  Nevertheless, we hope that the definitions and results of \S \ref{sec:spacing}-\ref{sec:densite}, and their rather unlikely origins in Neuroscience, can appear as worthy illustrations of the tremendous simplifications that symmetry can bring to this subject $-$ allowing one to bypass some of the hard analysis and geometry that one usually needs for concrete calculations.

{It is perhaps worth noting that much more challenging problems have recently been solved on some particular examples of symmetric spaces: see for instance the study of fluctutations of the nodal volume for the torus \cite{KKW}, or related problems on spheres and tori \cite{PeccaReel, Cammarota, RossiWigman, Rossi}. There symmetry arguments play a more discreet  but important role $-$ via Fourier analysis and the properties of certain special functions, which can be related to unitary representations by \S \ref{sec:generalites}. It would be interesting to know whether some parts of this deeper analysis can be understood in terms that may have a meaning for wider classes of spaces carrying a group action.   }

\subsubsection*{Acknowledgments} This paper is a substantially reworked version of a chapter in my Ph.D. thesis \cite{AAThese}, prepared at Universit\'e Paris-7 and the Institut de Math\'ematiques de Jussieu-Paris Rive Gauche.  I am deeply grateful to Daniel Bennequin for his advice and support. I thank Djalil Chafa\"i and Laure Dumaz for their more recent help, and the referees of successive versions for their very useful comments.


\section{Invariant fields on homogeneous spaces: general theory and decomposition theorems}
\label{sec:generalites}


\subsection{Gaussian fields and their correlation functions}
\label{subsec:gauss}

 Suppose $X$ is a smooth manifold and $V$ a finite-dimensional Euclidean space. Recall that a \emph{Gaussian random field on $X$ with values in $V$} is a random field $\mathbf{\Phi}$ on $X$ such that for each $n$ in $\mathbb{N}$ and every $n$-tuple $(x_{1}, ... x_{n})$ in $X^n$, the random vector $\left(\mathbf{\Phi}(x_{1}), ... \mathbf{\Phi}(x_{n})\right)$ in $V^n$ has a Gaussian distribution. A Gaussian field $\mathbf{\Phi}$ is \emph{centered} when the map $x \mapsto \mathbb{E}\left[ \mathbf{\Phi}(x) \right]$ is identically zero. It is \emph{continuous}, (resp. \emph{smooth}), when almost every sample map $x \mapsto \mathbf{\Phi}(x)$ is continuous (resp. smooth). It is \emph{qm-continuous}  where `qm' stands for `quadratic mean', when $\mathbb{E}\left[ \lvert \Phi(x)-\Phi(y)\rvert^2\right]$ goes to zero when $x$ goes to $y$ on $X$. It is \emph{qm-smooth} when the conditions discussed in \cite[\S I.4.3]{AzaisWschebor} are satisfied.
 
 The case in which $V$ equals $\mathbb{R}$ is of course important. If $\Phi$ is a real-valued Gaussian field on $X$, its \emph{covariance function} is the (deterministic) map $(x,y) \mapsto \mathbb{E}\left[ \Phi(x) \Phi(y) \right]$ from $X \times X$ to $\mathbb{R}$. A real-valued Gaussian field $\Phi$ is \emph{standard} if it is centered and if $\Phi(x)$ has unit variance at each $x \in X$.

 We shall work with real-valued fields in \S \ref{sec:spacing} and \S \ref{sec:densite}. But when describing fields with symmetry properties, the relationship with representation theory to be detailed in \S \ref{subsec:repthy}  (and used in \S \ref{sec:exemples}) makes it useful that the covariance function, and thus the field as well, be allowed to be complex-valued. A word about the case $V = \mathbb{C}$ is  therefore appropriate. 
  
 A \emph{circularly symmetric Gaussian variable in $\mathbb{C}$} is a complex-valued random variable whose real and imaginary parts are independent and identically distributed (real) Gaussian variables. A \emph{circularly symmetric complex Gaussian field on $X$} is a Gaussian centered random field $Z$ on~$X$ with values in the vector space $\mathbb{C}$, with the additional requirement that $(x,y) \mapsto \mathbb{E}\left[ Z(x) Z(y) \right]$ be identically zero: this condition imposes that $Z(x)$ be circularly symmetric for all $x$, but does not necessitate that $\mathfrak{Re}(Z)(x)$ and $\mathfrak{Im}(Z)(y)$ be uncorrelated if $x$ is not equal to $y$.

 Given  a circularly symmetric complex Gaussian field $Z$ on $X$, the \emph{correlation function} of~$Z$ is the (deterministic) map $(x,y) \mapsto \mathbb{E}\left[ Z(x) \bar{Z}(y) \right]$ from $X \times X$ to $\mathbb{C}$, where the bar denotes complex conjugation. A \emph{standard complex Gaussian field on $X$} is a {circularly symmetric complex Gaussian field on $X$} such that $\mathbb{E}\left[ Z(x) \bar{Z}(x) \right] = 1$ for all $x$. 

In order to relate the complex-valued case and the real-valued case, we note (as in \cite[\S 2.3]{Hida}) that 
\begin{itemize}
\item[$\bullet$] The real part of the correlation function of a standard complex-valued field $Z$ is twice the covariance function of the real-valued field $\mathfrak{Re}(Z)$.
\item[$\bullet$] A standard complex Gaussian field $Z$ has a real-valued correlation function if and only if its real and imaginary parts are independent \emph{as processes}.
\item[$\bullet$] Given a standard real-valued field $\Phi$ on $X$ with covariance function $C$, we can obtain a standard complex-valued field with correlation function $C$ by considering $\frac{1}{\sqrt{2}}\left(\Phi+i\tilde{\Phi}\right)$, where $\tilde{\Phi}$ is an independent copy of $\Phi$.
\end{itemize} 
 Let us now state the theorem which describes the correlation functions of standard complex Gaussian fields.

\begin{prop}\label{cova}
\begin{enumerate} 
\item Suppose $C$ is a deterministic map from $X \times X$ to $\mathbb{C}$. The map $C$ arises as the correlation function of a qm-continuous (resp. qm-smooth), invariant, standard complex-valued Gaussian field if and only if it has the following properties. 

\begin{enumerate}[(a)]
\item For every $x$ in $X$, we have $C(x,x) = 1$.
\item For each $n$ in $\mathbb{N}$ and every $n$-tuple $(x_{1}, ... x_{n})$ in $X^n$, the hermitian matrix $\left( C(x_{i}, x_{j}) \right)_{1 \leq i,j \leq n}$ is nonnegative-definite.
\item The map $C$ is continuous (resp. smooth).
\end{enumerate} 

\item If a map $C$ satisfies (a)-(c) above and if $\mathbf{\Phi}_{1}$ and $\mathbf{\Phi}_{2}$ are continuous (resp. smooth), invariant, standard complex-valued Gaussian fields with correlation function $C$, then $\mathbf{\Phi}_{1}$ and $\mathbf{\Phi}_{2}$ have the same probability distribution.\end{enumerate}
\end{prop}

This is extremely classical: see \cite[\S II.3]{Doob} or \cite[\S 2.3]{Hida}.  \qed

By the remarks above, Proposition \ref{cova} also describes the class of covariance functions of standard real-valued Gaussian fields. 

In relation with the continuity or smoothness of sample paths, we mention that the regularity condition \emph{in quadratic mean} of Proposition \ref{cova} can be replaced by the \emph{almost sure} regularity of sample paths under mild conditions; see \cite[Chapter 1, \S 4.3]{AzaisWschebor}. For instance, for the invariant fields to be discussed below, if the map $C$ is analytic, then it arises as the correlation function of a smooth (and not just q.m. smooth) field  \cite{Belyaev}.


\subsection{Invariant fields; relationship with representation theory}
\label{subsec:repthy}

Henceforth we will assume that the smooth manifold $X$ is equipped with a smooth and transitive action $(g, x) \mapsto g \cdot x$ of a Lie group $G$. A Gaussian field on $X$ with values in $V$ is \emph{invariant} (or \emph{homogeneous}) when the probability distribution of $\mathbf{\Phi}$ and that of the Gaussian field $\mathbf{\Phi} \circ \left( x \mapsto g \cdot x \right)$ are the same for every $g$ in $G$. 

The results of this paragraph and the next are due to Yaglom \cite{Yaglom}. Given the expository nature of the present sections, we will include proofs.

When $\Phi$ is an invariant complex-valued random field on $X$, the correlation function $C:X \times X \to \mathbb{C}$ satisfies 
\begin{equation} \label{covinv} \forall (x,y) \in X^2, \quad \forall g \in G, \quad C(gx, gy) = C(x,y).\end{equation}
We now choose $x_{0}$ in $X$ and write $H$ for the stabilizer of $x_{0}$ in $G$, so that $X$ is diffeomorphic with the coset space $G/H$.  Given a  map $C: X \times X \to \mathbb{C}$, the following two assertions are equivalent: 
\begin{itemize}
\item[$\bullet$] the map $C$ satisfies  \eqref{covinv}, satisfies parts (a)-(b) of Proposition \ref{cova}, and is continuous (resp. smooth),
\item[$\bullet$]  there exists a left-and-right $H$-invariant continuous (resp. smooth) function $\Gamma$ on $G$, taking the value one at $1_{G}$, such that $C(gx, x) = \Gamma(g)$ for every $g$ in $G$ and every $x$ in $X$.
\end{itemize}
 Proposition \ref{cova} thus says that taking covariance functions yields a natural bijection between
\begin{enumerate}[(i)] 
\item probability distributions of qm-continuous (resp. qm-smooth), invariant, standard, complex-valued Gaussian fields on $X = G/H$;
\item[and]
\item positive-definite, continuous (resp. smooth), $H$-bi-invariant functions on $G$, taking the value one~at~$1_{G}$. \end{enumerate}

A positive-definite, continuous, complex-valued function on $G$ which takes the value one at $1_{G}$ is usually called a \emph{state} of $G$. The invariant (smooth) Gaussian fields on $X$ thus correspond bijectively, at least if one identifies any two fields that share the invariant probability distributions, to the $H$-bi-invariant (and smooth) states of $G$. 

A crucial bridge between invariant random fields and unitary representation theory is the fact that unitary representations of $G$ are a natural source of states for $G$. Suppose $T$ is a unitary representation of $G$ on a Hilbert space $\mathcal{H}$ (recall that this means that $T$ is a continuous morphism from $G$ to the unitary group $U(\mathcal{H})$ of $\mathcal{H}$, equipped with the strong operator topology). Then for every unit vector $v$ in $\mathcal{H}$, the map $\Gamma: g \mapsto \langle v, U(g) v \rangle$ turns out\footnote{This is because for every $x_1, ... x_n$ in $G$, the hermitian matrix $\left( C(x_{i}, x_{j}) \right)_{1 \leq i,j \leq n}$ of Proposition \ref{cova}(d) is the Gram matrix associated with the vectors $T(x_1)v, T(x_2)v, ... T(x_n) v$ of the Hilbert space $\mathcal{H}$.
} to be a state of $G$.  In fact, when $G$ is unimodular, \emph{every} state of $G$ is attached with a unitary representation:

 \begin{prop}[Gelfand-Naimark, Segal]  Let $G$ be a unimodular Lie group. A function $m: G \to \mathbb{C}$ is a state of $G$ if and only if there exist a Hilbert space $\mathcal{H}$, a continuous morphism $T: G \to U(\mathcal{H})$, and a unit vector $v$ in $\mathcal{H}$, such that  $m(g) = \langle v, T(g) v \rangle$ for all $g$. Given such data, the function $m$ is $H$-bi-invariant if and only if the vector $v$ of $\mathcal{H}$ is invariant under all operators~$T(h)$, $h \in H$. 
  \end{prop}

 \begin{proof} {Building $\mathcal{H}$ and $T$ out of $m$ is the Gelfand-Naimark-Segal construction: on the vector space $\mathcal{C}_{c}(G)$ of continuous, compactly-supported functions on a  unimodular Lie group $G$, we consider the bilinear form 
\[ \langle f,g\rangle= \int_{G^2} m(x^{-1} y) f(x) \bar{f}(y) dxdy.\] It defines a scalar product on the quotient $\mathcal{C}_{c}(G)\big/ \! \left\{ f \in \mathcal{C}_{c}(G) \ , \ \langle f, f \rangle = 0 \right\}$, and we can complete this quotient into a Hilbert space $\mathcal{H}$; the left regular action of $G$ on $\mathcal{C}_{c}(G)$ then yields a unitary representation of $G$ on $\mathcal{H}$.  } 

{In order to build $v$ out of $m$, we remark that the linear functional $f \mapsto \int_{G} f \bar{m}$ determines a bounded linear functional on $\mathcal{H}$. The Riesz representation theorem yields one  $v$ in $\mathcal{H}$ which has the desired property. }

Given $(\mathcal{H}, T, v)$ such that $m(g)=\langle v, T(g) v\rangle$ for all $g$, the $T(H)$-invariance of $v$ of course implies the $H$-bi-invariance of $m$; in the reverse direction, if $m$ is bi-invariant, upon expanding the scalar product $\langle(T(h)-\text{id}_\mathcal{H})v,(T(h)-\text{id}_\mathcal{H})v\rangle$ one finds zero, so $v$ is $T(H)$-invariant.\end{proof}
 It is a consequence of parts (a)-(b) in Proposition \ref{cova}  that a state of $G$ is a bounded function on $G$ (and that its modulus does not exceed $1$). Thus, the $H$-bi-invariant states of $G$ form a convex subset $\mathcal{C}$ of $L^\infty(G)$. 
 
\subsection{Commutative spaces, elementary spherical functions and monochromatic fields}  \label{subsec:gelfand} We now assume that the group $G$ is unimodular and equip it with a bi-invariant Haar measure.  We can then view $L^\infty(G)$ as the dual of the space $L^1(G)$ of integrable functions and equip it with the weak topology; because of Alaoglu's theorem, the set $\mathcal{C}$ of states of $G$ then appears as a relatively compact, convex subset of $L^\infty(G)$. 

The extreme points of $\mathcal{C}$ are usually known as \emph{elementary spherical functions for the pair~$(G,H)$.} Their significance to representation theory is that they correspond to \emph{irreducible} unitary representations: if $m$ is a state of $G$ and $(\mathcal{H}, T, v)$ is such that $m$ is given by $g \mapsto \langle v , T(g) v \rangle$ as above, then $m$ is an elementary spherical function for $(G,H)$ if and only if the unitary representation $T$ of $G$ on $\mathcal{H}$ irreducible\footnote{Indeed, should there be $T(G)$-invariant subspaces $\mathcal{H}_{1}, \mathcal{H}_{2}$ such that $\mathcal{H} = \mathcal{H}_{1} \oplus \mathcal{H}_{2}$, orthogonal direct sum, writing $v = v_{1} + v_{2}$ with $v_{i}$ in $\mathcal{H}_{i}$, one would have $m(g) = \langle v_{1}, T(g) v_{1} \rangle + \langle v_{2}, T(g) v_{2} \rangle$, and each $g \mapsto \langle v_i, T(g) v_i \rangle$ would be a constant multiple of a state of $G$, thus $m$ would not be an extreme point of $\mathcal{C}$. For the reverse implication, use the Gelfand-Naimark-Segal construction.}. 

It is natural, in view of the Krein-Milman theorem, to expect that general $H$-bi-invariant states (points of $\mathcal{C}$) can be apprehended in terms of the elementary spherical functions (the extreme points of $\mathcal{C}$). This can be made precise when suitable conditions on the pair $(G,H)$, or alternatively on the space $X$, are satisfied.

\begin{defi} \label{commut} A smooth homogeneous space $X = G/K$ is \emph{commutative} when the following two conditions are satisfied:
\begin{itemize}
\item[$\bullet$] $K$ is a compact subgroup of $G$;
\item[$\bullet$] the convolution algebra $\mathbf{L}^1(K \backslash G / K)$ is commutative.
\end{itemize}
One says also that the pair $(G,K)$ is a \emph{Gelfand pair}.
\end{defi}

Then Choquet-Bishop-de Leeuw representation theorem, a measure-theoretic version of the Krein-Milman theorem, exhibits a general K-bi-invariant state as a ``direct integral'' of elementary spherical functions, in a way that mirrors the (initially more abstract) decomposition of the corresponding representation of $G$ into irreducibles. Let $\mathbf{\Lambda}$ be the space of extreme points of $\mathcal{C}$, a topological space if one lets it inherit the weak topology from $L^\infty(G)$. Then Choquet's theorem says every point of $\mathcal{C}$ is the barycentre of a probability measure concentrated on $\mathbf{\Lambda}$, and the probability measure is actually unique in our case: see \cite{Faraut}, Chapter~II.  

 We can summarize the above discussion with the following statement.

\begin{prop}[Godement-Bochner theorem] \label{boch} Suppose $(G, K)$ is a Gelfand pair, and~$\mathbf{\Lambda}$ is the (topological) space of elementary spherical functions for the pair $(G,K)$, or equivalently the (topological) space of equivalence classes of unitary irreducible representations of $G$ having a nonzero $K$-fixed vector. The $K$-bi-invariant states of $G$ are exactly the continuous functions $\gamma$ on $G$ which can be written as $\gamma = \int_{\mathbf{\Lambda}} \varphi_{\lambda} d\mu_{\gamma}(\lambda)$, where $\mu_{\gamma}$ is a probability measure on~$\mathbf{\Lambda}$.  \end{prop}

 Let us make the backwards way from the theory of positive-definite functions for a Gelfand pair~$(G,K)$ to that of Gaussian random fields on $G/K$. It starts with a remark: suppose $m_{1}, m_{2}$ are $K$-bi-invariant states of $G$ and $\Phi_{1}$, $\Phi_{2}$ are independent Gaussian fields whose covariance functions, when turned into functions on $G$ as before, are $m_{1}$ and $m_{2}$, respectively. Then a Gaussian field whose correlation function is ${m_{1} + m_{2}}$ necessarily has the same probability distribution as~$\Phi_{1}+\Phi_{2}$. For sums indexed by a larger parameter space, an application of Fubini's theorem extends this remark to provide a \emph{spectral decomposition for Gaussian fields}, which mirrors the above decomposition of spherical functions:

\begin{prop}[Yaglom] Let $X = G/K$ be a commutative space, and let $\mathbf{\Lambda}$ denote the space of elementary spherical functions for the pair $(G,K)$. 
\begin{itemize}
\item[$\bullet$] For every $\lambda$ in $\mathbf{\Lambda}$, there exists, up to equality of the probability distributions, exactly one Gaussian field whose correlation function is given by $\varphi_{\lambda}$; 

\item[$\bullet$] Suppose $\left( \Phi_{\lambda} \right)_{\lambda \in \Lambda}$ is a collection of mutually independent Gaussian fields, and that for each~$\lambda$, $\Phi_{\lambda}$ is a field whose correlation function is the elementary spherical function $\varphi_{\lambda}$. Then for each probability measure $\mu$ on $\mathbf{\Lambda}$,  the Gaussian field $\Phi$ defined by
$$\Phi(x) =  \int_{\mathbf{\Lambda}} \Phi_{\Lambda}(x) d\mu(\lambda), \quad x \in X$$
has correlation function $\int_{\mathbf{\Lambda}} \varphi_{\lambda} d\mu_{}(\lambda)$.
\end{itemize}\end{prop}
 In \S \ref{sec:exemples}, we shall focus on special cases over which it is possible to give explicit descriptions of~$\Lambda$ and, for each $\lambda$ in $\Lambda$, of the spherical function $\varphi_{\lambda}$ and of a Gaussian field whose correlation function is $\varphi_{\lambda}$. 


\subsection{Relationship with joint eigenvalues of the invariant differential operators}
\label{subsec:diffop}

Let us drop the commutativity hypothesis for a moment. If we are given a smooth homogeneous space $X=G/H$, then the existence of Gaussian fields on $X$ that are both \emph{smooth} and \emph{$G$-invariant} (or, in representation-theoretic terms, the existence of $H$-bi-invariant smooth states of~$G$) imposes constraints on $H$.

One of those is the existence of an invariant Riemannian metric on $X$. Assume that there exists a smooth, $G$-invariant and standard real-valued field $\Phi$ on $X$.  Then for each $p$ in $X$ and every $u,v$ in the tangent space $T_p X$,  we can set
\begin{equation} \label{riem} \eta_p(u,v) := \mathbb{E}\left[ (d\Phi(p) u)(d\Phi(p)v) \right]. \end{equation}
The expectation on the right-hand-side  has a meaning because the samples of $\Phi$ are almost surely smooth, and since $\Phi$ is standard and $G$-invariant,  formula \eqref{riem} induces a $G$-invariant Riemannian metric $\eta$ on $X= G/H$. Such a metric does not always exist; the simplest way for the existence to be guaranteed is for $H$ to be compact.

Even when $H$ is compact, the classification of $H$-bi-invariant smooth states of $G$ can be a difficult problem. But it becomes tractable (at least in principle) when $X$ is a commutative space. The Godement-Bochner theorem above says that one needs only classify the elementary spherical functions, and it turns out that the elementary spherical functions are solutions of systems of partial differential equations: they are joint eigenfunctions for all $G$-invariant differential operators on $X$. The next two statements will summarize what we need; see \cite[\S 8.3]{Wolf}.

\begin{prop}[Thomas] \label{thomas} Let $X = G/K$ be a smooth  homogeneous space, and let $\mathcal{D}_{G}(X)$ denote the algebra of $G$-invariant differential operators on $X$. Assume that $G$ is connected and that $K$ is compact. Then $X$ is a commutative space (Definition \ref{commut}) if and only if  the algebra~$\mathcal{D}_{G}(X)$ is commutative. 
\end{prop}

\begin{theo}[Gelfand, Godement, Helgason] \label{diffop} Suppose $X=G/K$ is commutative. A smooth, $K$-bi-invariant function $\phi: G\to \mathbb{C}$ is an elementary spherical function for $(G,K)$ if and only if there is, for each $D$ in $\mathcal{D}_{G}(X)$, a complex number $\chi_\varphi(D)$ such that the function $\varphi: X \to \mathbb{C}$ induced by $\phi$ satisfies
\begin{equation} \label{edp} D \varphi = \chi_{\varphi} (D) \varphi. \end{equation}
The eigenvalue assignment $D \mapsto \chi_{\varphi}(D)$ defines a character\footnote{Recall that a character of $\mathcal{D}_{G}(X)$ is a morphism of algebras from $\mathcal{D}_{G}(X)$ to $\mathbb{C}$.} of the commutative algebra $\mathcal{D}_{G}(X)$. It determines the spherical function $\phi$: for every character $\chi$ of $\mathcal{D}_{G}(X)$, there is at most one $K$-bi-invariant state $\phi$ of $G$ that satisfies  \eqref{edp}.
\end{theo}

Suppose $X$ is a smooth, commutative space. We wish to study standard Gaussian fields on~$X$; by Proposition \ref{cova}, it is enough to study their correlation functions; by Proposition \ref{boch}, that study can be reduced to the case when the covariance function is an elementary spherical function; by Theorem \ref{diffop}, the study of these spherical functions reduces to spectral theory for the algebra  $\mathcal{D}_{G}(X)$.

\begin{defi} A centered Gaussian random field on $X$ whose correlation function is an elementary spherical function will be called \emph{monochromatic}; the corresponding character of $\mathcal{D}_{G}(X)$ will be called its \emph{spectral parameter}. \end{defi}

A significant consequence  of Theorem \ref{diffop} and \S \ref{subsec:repthy}-\ref{subsec:gelfand} is that all statistical properties of a monochromatic field can, in principle, be expressed in terms of its spectral parameter. Suppose for instance that the commutative algebra $\mathcal{D}_{G}(X)$ is finitely generated. Then any character of~$\mathcal{D}_{G}(X)$ is specified by a finite collection of real numbers; to any monochromatic field, one can thus attach a finite collection of numbers, and the information about the field (for instance about its nodal volume) can in principle be expressed in terms of those numbers. 
 
We note that on a commutative space $X$, once we fix an invariant Riemannian metric, there is always in $\mathcal{D}_{G}(X)$ a distinguished element: the Laplace-Beltrami operator $\Delta_X$. In some important cases (like spheres, Euclidean spaces and hyperbolic spaces, and more generally two-point homogeneous spaces), the algebra $\mathcal{D}_{G}(X)$ is actually generated by $\Delta_X$, and \S \ref{subsec:repthy}-\ref{subsec:diffop} prove that the study of invariant random fields and their statistical properties reduces (in principle) to the spectral theory of the Laplacian.


\subsection{A class of examples in which which no invariant field can be continuous} \label{subsec:affine}

In the two previous paragraphs, we saw that when $X$ is a commutative space, a precise enough knowledge of representation theory, or of the spectral theory of invariant differential operators on $X$, can lead to the construction of ``many'' smooth invariant fields on $X$. The remark at the beginning of \S \ref{subsec:diffop} shows that this favorable situation cannot be expected if $X$ is a more general homogeneous space. We would like to point out here that it can happen, on some homogeneous spaces, that even the weaker property of sample path continuity cannot be satisfied by any invariant field. 

Our class of examples will consist of affine spaces: we fix a finite-dimensional vector space $\mathfrak{v}$ and a closed subgroup $H$ of $\text{GL}(\mathfrak{v})$. Write $G = H \ltimes \mathfrak{v}$ for the group of all affine transformations of $\mathfrak{v}$ whose linear part lies in $H$. Suppose we want to study $G$-invariant Gaussian fields on the flat space $\mathfrak{v}=G/H$. 

Let us assume that we are given such a $G$-invariant Gaussian field $\Phi: \mathfrak{v} \to \mathbb{C}$ and that $\Phi$ has continuous sample paths. By Proposition \ref{cova}, the correlation function of $\Phi$  is given by a continuous, bounded and $H$-bi-invariant state $\Gamma$ of $G$. A function on the semidirect product $G=H \ltimes \mathfrak{v}$ that is right-invariant under $H$ is entirely determined by its restriction to $\mathfrak{v}$, so all the information about $\Phi$ is in principle encoded by the restriction $\gamma=\Gamma_{|\mathfrak{v}}$.  We can study $\gamma$ through its Fourier transform $\widehat{\gamma}$, which is a tempered distribution on the vector space dual $\mathfrak{v}^\star$ of $H$. The fact that $\gamma$ is also left-invariant under $H$ implies that $\widehat{\gamma}$ is invariant under the action of $H$ on $\mathfrak{v}^\star$. 

The geometry of the action of $H$ on $\mathfrak{v}^\star$ thus encodes much information about the possible $G$-invariant Gaussian fields on $X$. In some cases, the geometry can be an obstruction to the existence of any continuous field:

\begin{prop}  \label{semid}
Let $G$ be a semidirect product $H \ltimes \mathfrak{v}$ as above. If there is no compact orbit of $H$ in $\mathfrak{v}^\star$ but the trivial one, then the only real-valued standard Gaussian field on $\mathfrak{v}$ whose probability distribution is $G$-invariant is the one whose samples are a.s. constant.
\end{prop}

\begin{proof} Such a field would yield a positive-definite, continuous, $H$-bi-invariant function, say~${\Gamma}$, on $G$, taking the value one at $1_{G}$. We shall see now that there can be no such function except the constant one. Write $\gamma$ for the restriction $\Gamma_{|\mathfrak{v}}$. The map $\gamma$ is nonnegative-definite and $H$-invariant; by Bochner's theorem (see for instance \cite[\S I.2]{Hida}), it is the Fourier transform of a bounded measure $\nu_{\gamma}$ on $\mathfrak{v}^\star$, and the measure $\nu_{\gamma}$ must also be $H$-invariant. Now, suppose $\Omega$ is a compact subset of an $H$-orbit in $\mathfrak{v}^\star$; then $\nu_{\Gamma}(\Omega)$ must be equal to $\nu_{\Gamma}(h \cdot \Omega)$ for each $h$ in $H$. If the orbit $H \cdot \Omega$ is noncompact, then there is a sequence $(h_{n})$ in $H^\mathbb{N}$ such that $\cup_{n} h_{n} \cdot \Omega$ is a disjoint union, and so the total mass of $\nu_{\Gamma}$ cannot be finite unless $\nu_{\Gamma}(\Omega)$ is zero. A consequence is that the support of $\nu_{\Gamma}$ must be the union of compact orbits of $H$ in $\mathfrak{v}^\star$. If the only compact orbit is the origin, then $\nu_{\Gamma}$ must be the Dirac distribution at the origin, and $\Gamma$ must be identically one. Proposition \ref{semid} follows.
 \end{proof}
 
 \begin{exem} Suppose $\mathfrak{v} \simeq \mathbb{R}^4$ is Minkowski spacetime and $H$ is the Lorentz group $SO(3,1)$. Then $G = H \ltimes \mathfrak{v}$ is the Poincar\'e group, and the hypothesis of Proposition \ref{semid} is satisfied (all orbits of $H$ in $\mathfrak{v}^\star$ are noncompact quadrics).  As a consequence, no Gaussian random field on Minkowski spacetime can be both continuous and invariant under the isometry group of $\mathfrak{v}$. \end{exem}


\section{Explicit description of smooth invariant fields on some Riemannian homogeneous spaces}

The aim of this section is to give important examples of Riemannian homogeneous spaces over which all smooth invariant fields can be described in a very concrete manner, by bringing together the general results of \S \ref{sec:generalites} (essentially due to Yaglom) with  available facts from representation theory (most of which are extremely classical). All our examples will belong to the special class of commutative spaces.
 
We begin with a homogeneous space $X = G/H$. If $X$ is commutative, then there exists on~$X$ a $G$-invariant Riemannian metric, and every such metric has constant scalar curvature. The difficulty of harmonic analysis on $X$ depends in a spectacular manner on the sign of the curvature: it is quite simple if the curvature of $X$ is nonnegative, and quite a bit more complicated if $X$ is negatively curved. 
 
\label{sec:exemples} 


\subsection{Flat spaces}
\label{subsec:plat} Suppose $X=G/K$ is a \emph{flat} commutative space. Then we know from early work by J.~A.~Wolf (see \cite[\S 2.7]{WolfSCC}) that $X$ is isometric to a product $\mathfrak{v} \times \mathbb{T}^s$ between a Euclidean space and a flat torus. 

We shall consider, mainly for simplicity, the simply connected case. Let $\mathfrak{v}$ be a Euclidean space, let $K$ be a closed subgroup of the orthogonal group $O(\mathfrak{v})$, and let $G$ be the semidirect product $K \ltimes \mathfrak{v}$. Let us consider the commutative space $X = \mathfrak{v}=G/K$. By applying the theory of \S \ref{sec:generalites}, we can describe the $G$-invariant continuous Gaussian fields on $\mathfrak{v}$ from the monochromatic ones, or equivalently, from the elementary spherical functions for the pair $(G, K)$.

The ideas of \S \ref{subsec:affine} then lead to a straightforward description of the elementary spherical functions:

\begin{prop} \label{sphmac} Suppose $\mathfrak{v}$ is a Euclidean vector space, and $K$ is a closed subgroup of $O(\mathfrak{v})$. Let $\Omega$ be a $K$-orbit in $\mathfrak{v}^\star$. Let $T_\Omega$ be the tempered distribution on $\mathfrak{v}$ that arises as the inverse Fourier transform of the Dirac distribution on $\Omega$. Then $T_\Omega$ is given by integration against an analytic function $\psi_\Omega$. Let  $\varphi_\Omega$  be the only $K$-bi-invariant function on $G$ whose restriction  to $\mathfrak{v}$ is~$\psi_{\Omega}$. Then $\varphi_\Omega$ is a constant multiple of an elementary spherical function for the pair $(G, K)$.

In the reverse direction, suppose $\varphi$ is an elementary spherical function for $(G,K)$. Then there exists an orbit $\Omega$ of $K$ in $\mathfrak{v}^\star$ such that the restriction of $\varphi$ to $\mathfrak{v}$ is a constant multiple of $\psi_\Omega$.  \end{prop}

\begin{proof} We first remark that if $\mu$ is a bounded measure on the orbit space $\mathfrak{v}^\star/K$ and the measure $\tilde{\mu}$ on $\mathfrak{v}^\star$ obtained by pulling $\mu$ back with the help of the Hausdorff measure of each $K$-orbit (normalized so that each orbit has total mass one), then the inverse Fourier transform of $\tilde{\mu}$ provides a positive-definite function for $(G, K)$. Indeed, Bochner's theorem says the inverse Fourier transform provides a positive-definite function on $\mathfrak{v}$; by extending that function to a $K$-right-invariant function on $G$, we obtain\footnote{The $K$-bi-invariance is immediate; for the nonnegative-definiteness, we use the fact that if $(k_{1}, v_{1})$ and $(k_{2}, v_{2})$ are elements of $G= K \ltimes V$, then the group product $(k_{1}, v_{1})(k_{2}, v_{2})^{-1}$ is equal to $(k_{1}k_{2}^{-1}, v_{1} -k_{1} k_{2}^{-1} v_{2})$; as a consequence, a $K$-bi-invariant function takes the same value at $(k_{1}, v_{1})(k_{2}, v_{2})^{-1}$ as it does at $(k_{1}^{-1}, 0)(k_{1}, v_{1})(k_{2}, v_{2})^{-1}(k_{2}, 0) = (1_{K}, k_{1}^{-1} v_{1} -k_{2}^{-1} v_{2})$; so if we start from a positive-definite function on $\mathfrak{v}$, we do obtain a positive-definite function on $K \ltimes \mathfrak{v}$. } a $K$-bi-invariant state of $G$. 

Now suppose ${\varphi}$ is a $K$-bi-invariant state of $G$, and let $\gamma$ be its restriction to $V$; this is a bounded $K$-invariant positive-definite continuous function on $V$. The Fourier transform of $\gamma$ is a bounded complex measure $\mu_\gamma$ on $\mathfrak{v}^\star$ with total mass one (because of Bochner's theorem). The measure  is $K$-invariant, and so it yields a bounded measure $\mu_\gamma$ on the orbit space $\mathfrak{v}^\star/K$. If the support of $\mu_\gamma$ is not a singleton, then we can split it into a half-sum of two bounded measures with total mass one and disjoint supports; lifting these two to $\mathfrak{v}^\star$ and taking Fourier transform exhibits our initial state as a sum of two $K$-bi-invariant functions taking the value one at $1_{G}$, and these two are positive-definite because of the previous remarks. So $\varphi$ is an extreme point among the $K$-bi-invariant states if and only if it corresponds, under restriction and Fourier transform, to a $K$-invariant measure concentrated on a single $K$-orbit in $\mathfrak{v}^\star$. This proves the proposition. \end{proof}
\begin{exem} If $X$ is ordinary Euclidean space, so that $\mathfrak{v}=\mathbb{R}^n$ and $K= O(n)$, then the $K$-orbits in $\mathfrak{v}^\star$ are the spheres centered at the origin. Proposition \ref{sphmac} says that the elementary spherical functions are (up to multiplicative constants) the Bessel functions $J_R (x)=  \int_{\mathbf{S}^{n-1}} e^{i \langle R u, x\rangle} du$, $R>0$. \end{exem}

For an affine commutative space $X$, Proposition \ref{sphmac} is a completely explicit description of the elementary spherical functions for $(G,K)$; in fact it is concrete enough to make it easy to simulate monochromatic Gaussian fields on $X$ on a computer. Indeed, suppose $\Omega$ is a $K$-orbit in $\mathfrak{v}^\star$, and suppose we wish to describe the (unique) standard Gaussian field $\Phi_{\Omega}$ on $\mathfrak{v}$ whose covariance function is the elementary spherical function $\varphi_{\Omega}$ of Proposition \ref{sphmac}. By definition, we must have $\mathbb{E}\left[ \Phi_{\Omega}(x) \Phi_{\Omega}(0) \right] = \varphi_\Omega(x)$ for all $x$, and as a consequence almost all samples of $\Phi$ must have their Fourier transform concentrated on $\Omega$. Thus, the samples $\Phi$ are superpositions of waves whose wave-vectors lie on $\Omega$, but have various directions and phases. 

At this point, it seems one can build a realization for $\Phi$ by independently assigning, to each wavevector  $\omega \in \Omega$, a random weight $\zeta(\omega) \in \mathbb{C}$ according to a standard Gaussian distribution, then adding up everything to obtain $\Phi(x) =  \int_{\Omega} e^{i\langle \omega, x\rangle} \zeta(\omega) d\omega$. The idea of using a family of independent variables indexed by an uncountable set $\Omega$ raises technical difficulties of the kind discussed in \cite[\S1.4.3 and \S 5.2]{AdlerTaylor}; we shall thus formulate the idea using the stochastic integral of \cite[\S 5.2]{AdlerTaylor}. Recall that $\Omega$ is a compact submanifold of $\mathfrak{v}^\star$; the Euclidean structure of $\mathfrak{v}^\star$ induces on~$\Omega$ a Riemannian metric, and thus a volume form with finite total mass.

\begin{lemm} \label{Mack} Let $\mu$ be the bounded measure on $\Omega$ inherited from the Euclidean structure of~$\mathfrak{v}^\star$, normalized so as to have total mass one. Let $\mathbf{Z}$ be the Gaussian $\mu$-noise on $\Omega$ introduced in  \cite[\S1.4.3]{AdlerTaylor}.  Then the Gaussian random field on $\mathfrak{v}$ defined, using the stochastic integration of~\cite[\S 5.2]{AdlerTaylor}, by
\begin{equation} \label{adler} \Phi_{\Omega}(x) =  \int_{\Omega} e^{i \langle \omega, x\rangle} d\mathbf{Z}(\omega), \quad x \in \mathfrak{v}, \end{equation}
is $G$-homogeneous, smooth, and has covariance function $\varphi_{\Omega}$. \end{lemm}

\begin{proof} Once $\Phi_{\Omega}$ is defined using the stochastic integral \eqref{adler}, the Lemma follows from \cite[Eq.~(5.2.9)]{AdlerTaylor}.

 \end{proof}

The spectral representation in \eqref{adler} is very concrete (see Figure \ref{fig1}): given points $\omega_1, ..., \omega_n$ on $\Omega$ and a family $\zeta_1, ..., \zeta_n$ of  i.i.d standard Gaussian variables in $\mathbb{C}$, the random map $x \mapsto \frac{1}{n} \sum \limits_{k=1}^n \zeta_k \cdot e^{i \langle \omega_k, x\rangle}$ can give a good approximation for $\Phi$, provided $n$ is ``sufficiently large'' and $\omega_1, ..., \omega_n$ are ``sufficiently uniformly distributed''  on $\Omega$. We omit a precise formulation (see \cite[\S 5.2]{AdlerTaylor}). 

\begin{figure}[h]\label{fig1}
\begin{center}
\includegraphics[width=0.3\textwidth]{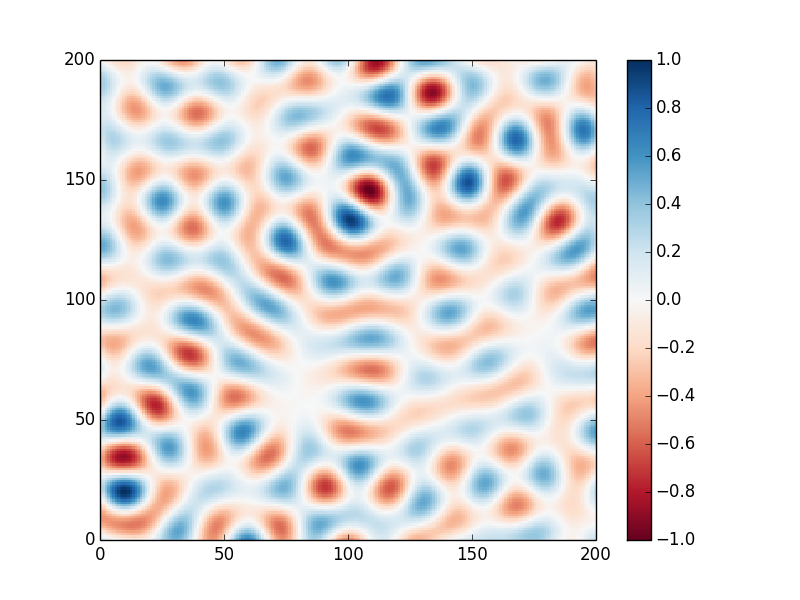}
 \includegraphics[width=0.3\textwidth]{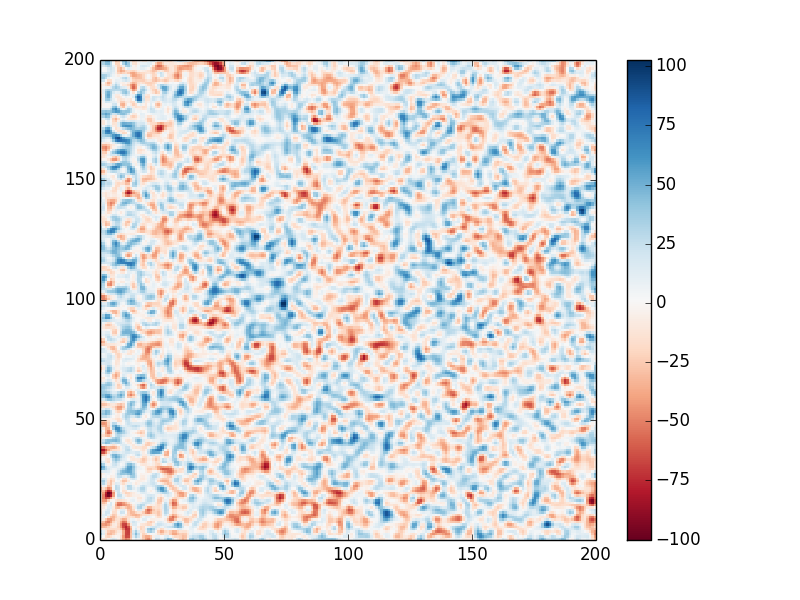} 
 \includegraphics[width=0.3\textwidth]{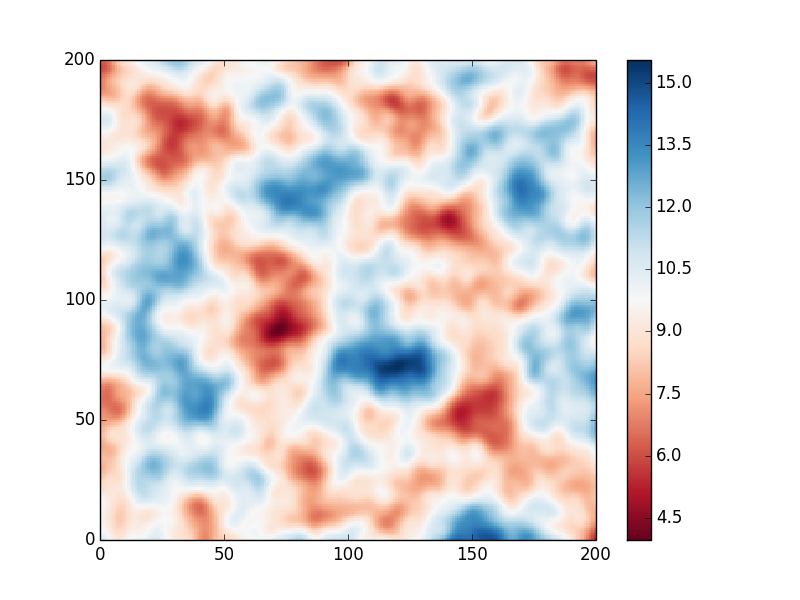}
 \end{center}
\caption{These are three real-valued maps on the plane, each of which is sampled from a real-valued invariant field for the Euclidean motion group. The field from which the map on the left is sampled is monochromatic. Because Proposition \ref{sphmac}, the elementary spherical functions form a half-line; the power spectrum (the measure on $\mathbb{R}^+$ defined in Proposition \ref{boch}) for the  map in the middle is roughly the indicatrix of a segment. The power spectrum for the map on the right has the same support but decreases as $1/R^2$. }
\end{figure}


\subsection{Positively-curved spaces}\label{subsec:spherique}
Suppose the commutative space $X = G/K$ is \emph{positively curved}, and assume for simplicity that $G$ is connected. Then $G$ is a connected \emph{compact} Lie group. 

 In that case, the representation theory of $G$ is extremely classical. The Hilbert spaces for irreducible representations of $G$ are finite-dimensional; so if $T : G \rightarrow U(\mathcal{H})$ is an irreducible representation, one can define a map $\chi_T: G \to \mathbb{C}$  by setting $\chi_T(g)= \text{Trace}(T(g))$ for all $g$ in $G$. The map $\chi_T$  is a {continuous} function on $G$, called  the \emph{global character} of $T$. 

Given an irreducible representation $T$ of $G$, we may form the map $\varphi_T: G \to \mathbb{C}$ defined by
\begin{equation} \label{VanDijk} \forall x \in G, \ \varphi_{T}(x)= \ x \mapsto \int_{K} \chi_T(x^{-1} k) dk \end{equation}
 where $\chi_T$ is the global character of $T$ and and the integration is performed w.r.t the normalized Haar measure of $K$.

\begin{prop}[\cite{VanDijk}, Theorem 6.5.1] \label{vandi} Let $G$ be a connected compact Lie group and $K$ be a closed subgroup of $G$. \begin{enumerate}
\item Suppose $T$ is an irreducible representation of $G$.  Then either the map $\varphi_T$ of \eqref{VanDijk} is an elementary spherical function for $(G,K)$, or $\varphi_T$ is identically zero. It is nonzero if and only if $T$ admits a nonzero $K$-fixed vector.
\item Every elementary spherical function $\varphi$ for $(G,K)$ reads $\varphi=\varphi_T$ for some irreducible representation $T$~of~$G$.
\end{enumerate} \end{prop}

A consequence of Proposition \ref{vandi} is that for compact $G$, obtaining a concrete description  of the elementary spherical functions for $(G,K)$ is easy if one has a precise knowledge of the global characters of irreducible representations. But there are extremely classical formulas (due to Hermann Weyl) for these characters, in terms of a maximal torus of $G$, and of combinatorial data $-$ roots and weights and highest weights. These can be evaluated concretely in many examples (see the comprehensive book \cite{Vilenkin}). 

Let us turn to a concrete description of the monochromatic Gaussian fields rather than their correlation functions. The analogy between \eqref{VanDijk} and the Fourier picture of Proposition \ref{sphmac} leads to an analogue of Lemma \ref{Mack}:

 \begin{lemm} \label{Vdijk} Let $\mu$ be the normalized Haar measure of $K$, and let $\mathbf{Z}$ be the Gaussian $\mu$-noise on $K$ introduced in  \cite[\S1.4.3]{AdlerTaylor}. Let $T$ be an irreducible representation of $G$. Then the Gaussian random field $\Phi_{T}$ on $X=G/K$ defined, using the stochastic integral of \cite[\S 5.2]{AdlerTaylor}, by
$$\Phi_{T}(gK) =   \int_{K}   \chi_T(g^{-1} k)  d\mathbf{Z}(k), \quad x \in G, $$
is $G$-homogeneous, smooth, and has covariance function $\varphi_{T}$. \end{lemm}

\begin{proof} Use \cite[Eq. (5.2.9)]{AdlerTaylor} and the Schur-Weyl orthogonality relations.\end{proof}

Let $T : G \rightarrow U(\mathcal{H})$ be an irreducible representation of $G$ that admits a nonzero $K$-fixed vector. Since we assumed $X$ to be commutative, the space  of $K$-fixed vectors in $\mathcal{H}$ actually has dimension one (see \cite[\S 6.3]{VanDijk}). Let $(\mathbf{e}_{1}, ... \mathbf{e}_{d})$ be an orthonormal basis for $\mathcal{H}$ whose first vector is $K$-invariant.

For every $i$, $j$ in $\{1, ..d\}$, we can consider the \emph{matrix element} $f_{ij}: G \to \mathbb{C}$ defined by $f_{ij}(g)=\langle \mathbf{e}_{i}, T(g) \mathbf{e}_{j} \rangle_{\mathcal{H}}$ for all $g \in G$.

\begin{prop}[\cite{Yaglom}] \label{yag} \begin{enumerate}[(i)] \item The function $f_{11}: g \mapsto \langle \mathbf{e}_{1}, T(g) \mathbf{e}_{1} \rangle_{\mathcal{H}} $ is an elementary spherical function for $(G,K)$; in fact  $f_{11}$ and the function $\varphi_T$ of \eqref{VanDijk} coincide. 

\item  Suppose $(\zeta_{i})_{i = 1, ... d}$ is a collection of i.i.d. standard Gaussian variables. Then the random function
\begin{equation} \label{Yaglom} gK \mapsto \sum \limits_{i= 1}^{d} \zeta_{i} \cdot f_{i1}(g) \end{equation}
is an invariant standard Gaussian random field on $G/K$, whose covariance function is $\varphi_T$. \end{enumerate}
\end{prop}

Yaglom in fact proves  results analogous to (i)-(ii)  when $X=G/K$ is a homogeneous space with compact $G$ (not necessarily a commutative one), but we will omit the more general description. \\

For compact commutative spaces, Proposition \ref{vandi} and Proposition \ref{yag} provide two different descriptions of the same spherical function and of the corresponding field. The description of Proposition \ref{vandi} is closer in spirit to that of the previous paragraph (and of the next). An advantage of Yaglom's compact-specific picture in \eqref{Yaglom} is that, in some important cases, very explicit bases $(\mathbf{e}_{i})$ and very explicit formulae for the matrix elements $g \mapsto \langle \mathbf{e}_{i}, T(g) \mathbf{e}_{j} \rangle$ are known: the obvious reference is \cite{Vilenkin}. For instance, if $X$ is the two-sphere, then the matrix elements $f_{ij}$ can be chosen to be classical spherical harmonics, and in that case \eqref{Yaglom} is perhaps a more concrete description of the invariant fields than \eqref{VanDijk}. 

We remark that Baldi, Marinucci and Varadarajan proved relatively recently \cite{BaldiMarinucci} that  \eqref{Yaglom} is in a sense the only way to build an invariant Gaussian field from a linear combination of the matrix elements $f_{ij}$.
\begin{figure}\label{fig2}
\begin{center}

\includegraphics[width=0.35\textwidth]{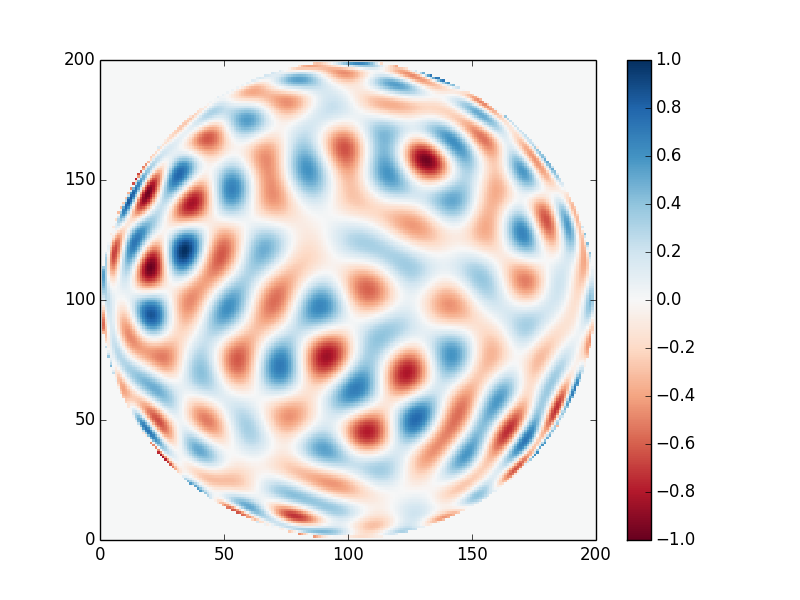}

\caption{A sample from a real-valued monochromatic field on the sphere. This uses a combination of spherical harmonics with i.i.d. Gaussian coefficients.}
\end{center}
\end{figure}
%


\subsection{Negatively-curved spaces: the particular case of symmetric spaces of the noncompact type}
\label{subsec:hyperbolique}

Suppose the commutative space $X = G/K$ is \emph{negatively curved}. Then $G$ cannot be compact; without any additional hypothesis on $G$ it can be quite difficult to do geometry and analysis on $X$, and the representation theory of $G$ can be quite wild. The situation is much better understood if $X$ is a \emph{Riemannian symmetric space of the noncompact type}: the isometry group of $X$ is then a semisimple Lie group, and the representation theory of $G$ is a vast, classical and deep subject.

 In that case, much is known about the geometry of $G/K$ and about the structure of the algebra $\mathcal{D}_{G}(X)$. Harish-Chandra determined the elementary spherical functions for $(G,K)$ in 1958; Helgason later reformulated his discovery in a way which brings it very close to Proposition \ref{sphmac}. For the contents of this subsection, see chapter III in \cite{HelgasonGASS}. We shall use much standard notation and refer to \cite[Ch. V-VII]{KnappPrincetonBook} or \cite[Ch. VI-VII]{KnappLieGroups} for background.

 Suppose $G$ is a closed connected subgroup of $GL(n, \mathbb{R})$ for some $n>0$, and assume that $G$ is stable under transpose and has finite center. Let $K$ be the set of orthogonal matrices in $G$; the group $K$ is a maximal compact subgroup of $G$ and it is the symmetric space $X=G/K$ that we shall study. Let $\mathfrak{g}$ denote the Lie algebra of $G$, let $\mathfrak{a}$ denote the set of diagonal matrices in $\mathfrak{g}$ and $A$ denote the subgroup $\exp(\mathfrak{a})$ of $G$. Let $\mathfrak{n}$ denote the set of block upper triangular matrices in $\mathfrak{g}$ whose diagonal blocks are zero, and let $N$ denote the subgroup $\exp(\mathfrak{n})$ of $G$. By the Iwasawa decomposition, every element $g$ in $G$ can be written uniquely as a product $k e^{H} n$ in which $k$ lies in $K$, $H$ lies in $\mathfrak{a}$ and $n$ lies in $N$. The map  $\mathfrak{A}: G \to \mathfrak{a}$ which takes $x=ke^{H}n$ to $\mathfrak{A}(x) = H$ is smooth.  Besides, if $H$ lies in $\mathfrak{a}$ and $U$ lies in $\mathfrak{n}$, then $HU-UH$ lies in $\mathfrak{n}$. Define a linear form $\rho:  \mathfrak{a}\to \mathbb{R}$ by setting $\rho(H) = \frac{1}{2}\left(\text{trace of the endomorphism $U \mapsto HU-UA$ of  $\mathfrak{n}$}\right)$. 
  
 Now suppose $\lambda$ is in $\mathfrak{a}^\star$ and $b$ is in $K$. Define
\begin{align*} e_{\lambda,b} : G \rightarrow & \ \mathbb{C} \\
x \mapsto & \  e^{\braket{i\lambda + \rho \ | \ \mathfrak{A}(b^{-1} {x})}}. \end{align*}
This is a smooth and right-$K$-invariant function from $G$ to $\mathbb{C}$, and thus it induces a function, for which we will also write $e_{\lambda,b}$, from $X = G/K$ to $\mathbb{C}$. This map is an eigenfunction of $\Delta_{X}$, with eigenvalue\footnote{Here the norm is the one induced by the Killing form.} $-\left(\norme{|\lambda}^2 + \norme{\rho}^2\right)$. We shall call $e_{\lambda,b}$ a \emph{Helgason wave} on $X$: it  is useful for harmonic analysis on $X$ in about the same way as plane waves are for classical Fourier analysis.

There are relations between the various Helgason waves, about which we now say a word.

If $(\lambda_{1}, b_{1})$ and $(\lambda_{2}, b_{2})$ are elements of $\mathfrak{a}^\star \times K$, then $e_{\lambda_{1},b_{1}}$ and $e_{\lambda_{2},b_{2}}$ coincide if and only if there is an element $k \in K$ such that (a) conjugation by $k$ preserves the property of being diagonal, (b) $\lambda_{1}\circ \left( \text{conjugation by $k$} \right) = \lambda_{2}$, and (c) $b_{1} = k b_{2}$.

Let $M$ be the subgroup of $K$ that consists of those orthogonal matrices $k \in K$ for which conjugation by $k$ acts trivially on $\mathfrak{a}$: the elements of $M$ are block-diagonal matrices with orthogonal diagonal blocks. Among the elements of $\mathfrak{a}$, say that an element  $H$ is \emph{regular} when the subgroup $\left\{ k \in K \ : \  k^{-1} H k = H\right\}$ of $K$ is equal to $M$. Let $\mathfrak{a}^\star$ be the vector space dual of $\mathfrak{a}$; if we use the (nondegenerate) Killing form to identify $\mathfrak{a}^\star$ with $\mathfrak{a}$, we obtain a notion of regular element in $\mathfrak{a}^\star$. It turns out that the set of regular elements in $\mathfrak{a}^\star$ is the complement of a finite number of hyperplanes. Among the connected components of the set of regular elements, each of which is an open cone in $\mathfrak{a}^\star$, the relationship between $\mathfrak{a}$ and $\mathfrak{n}$ singles out a component $\mathcal{C} \subset \mathfrak{a}$ (see \cite[Chap. VII]{KnappPrincetonBook}). Each of the $e_{\lambda,b}$ coincides with one of the $e_{\lambda^+,b^+}$, where $\lambda^+$ runs through the closure $\Lambda^+$ of $\mathcal{C}$ in $\mathfrak{a}^\star$ and $b^+$ runs through $K$.
 
 We come now to Harish-Chandra's description of the elementary spherical functions for $(G,K)$ in terms of Helgason waves.

\begin{theo}[Harish-Chandra]
For each $\lambda$ in $\Lambda^+$, the map
$$ \varphi_{\lambda} := x \mapsto \int_{K} e_{\lambda, b}(x) db $$
is\footnote{Here the invariant measure on $K$ is normalized so as to have total mass one.} an elementary spherical function for $(G,K)$. Every spherical function for $(G,K)$ is one of the  $\varphi_{\lambda}$, $\lambda \in \Lambda^+$.
\end{theo}

 Thus the possible spectral parameters for monochromatic fields occupy a closed cone $\Lambda^+$ in the Euclidean space $\mathfrak{a}^\star$ (and the topology on the space of spherical functions described in \S \ref{subsec:gelfand} coincides with the topology inherited from $\mathfrak{a}^\star$). 
 
 Part of the above theorem says that spherical functions here again appear as a constructive interference of waves propagating in various directions. This yields an explicit description for the monochromatic field with spectral parameter $\lambda$, in the spirit of (and with the same proof as) Lemma \ref{Mack}:

\begin{lemm} \label{Helgason} Let $\mu$ be the normalized Haar measure of $K$, and let $\mathbf{Z}$ be the Gaussian $\mu$-noise on $K$ introduced in  \cite[\S1.4.3]{AdlerTaylor}. Then the Gaussian random field on $X=G/K$ defined, using the stochastic integral of \cite[\S 5.2]{AdlerTaylor}, by
$$\Phi_{\lambda}(gK) = \int_{K} e_{\lambda,b}(g) d\mathbf{Z}(b), \quad g \in G, $$
is $G$-homogeneous, smooth, and has covariance function $\varphi_{\lambda}$. \end{lemm}

\begin{figure}[!h] \label{fig3}
\begin{center}

\includegraphics[width=0.35\textwidth]{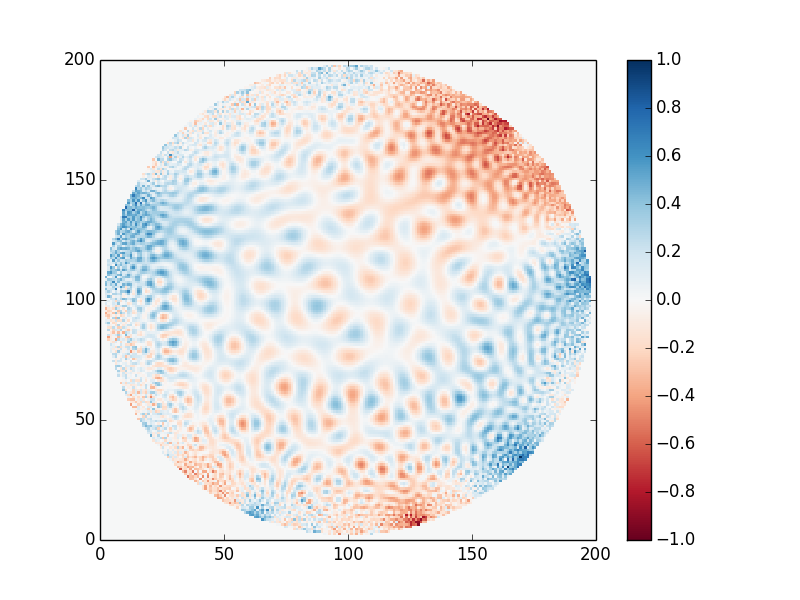}

\caption{A real-valued map on the Poincar\'e disk, sampled from a monochromatic field using Lemma \ref{Helgason}. }
\end{center}
\end{figure}


\section{The typical spacing in a real-valued invariant field}
\label{sec:spacing}

Let us start with a homogeneous \emph{real}-valued Gaussian field $\Phi$ on a Riemannian homogeneous space $X$. On the above pictures drawn on two-dimensional Riemannian symmetric spaces, we see that when the correlation function of $\Phi$ is close enough to being an elementary spherical function, one may expect $\Phi$ to exhibit some form of quasiperiodicity (we use the word ``quasiperiodic'' in a loose sense here, not a mathematically precise one). The previous pictures are drawn, of course, on very special (and commutative) spaces, but we shall henceforth work in the general setting of a general homogeneous space $X=G/K$ endowed with a left-$G$-invariant metric.

Let us now see whether we can give a meaning to the ``quasiperiod'' in such a general situation. Draw a geodesic $\gamma$ on $X$, and if $\Sigma_{}$ is a segment on $\gamma$, write $\mathcal{N}_{\Sigma}$ for the random variable recording the number of zeroes of $\Phi$ on $\Sigma$. Because the field $\Phi$ is homogeneous and the metric on $X$ is invariant, the probability distribution of $\mathcal{N}_{\Sigma}$ depends only on the length, say $|\Sigma|$, of $\Sigma$. The identity component of the subgroup of $G$ fixing $\gamma$ is a one-parameter subgroup of $G$, and reads $\exp_{G}(\mathbb{R} \vec{\gamma})$ for some $\vec{\gamma}$ in $\mathfrak{g}$ ; it is isomorphic either to a circle or to the additive group of the real line \footnote{In the particular case where  $X$ is a Riemannian symmetric space, the subgroup under discussion is a circle if $X$ is of the compact type, and a line if $X$ is of the Euclidan or noncompact type.}.

As a consequence, we can pull back $\Phi|_{\gamma}$ to the line $\mathbb{R}\vec{\gamma}$ in $\mathfrak{g}$ and view it as a stationary, real-valued Gaussian field on the real line. In this way, the group exponential relating $\mathbb{R}\vec{\gamma}$ to $\gamma$ sends the Lebesgue measure of $\mathbb{R}$ to a constant multiple of the metric which $\gamma$ inherits from that of $X$. The zeroes of the pullback of $\Phi|_{\gamma}$ to $\mathbb{R}\vec{\gamma}$ can thus be studied through the classical, one-dimensional,  Kac-Rice formula:

\begin{prop}[Rice's formula for the level zero] Suppose $\mathbf{u}$ is a translation-invariant smooth and centered Gaussian field on the real line. Consider an interval $I$ of length $\ell$ on the real line. Write $\mathcal{N}_{I}$ for the random variable recording the number of points $x$ on $I$ where $\mathbf{u}(x) = 0$ ; then
\begin{equation}\label{KacRice} \mathbb{E}\left[ \mathcal{N}_{I} \right] = \ell \cdot  \frac{ \sqrt{\kappa} }{\pi} \end{equation}
where $\kappa = \mathbb{E}\left[\mathbf{u}'(0)^2 \right]$ is the second spectral moment of the field. 
\end{prop}

Now, recall that our current objective is to give a meaning to the ``quasiperiod'' of a field $\Phi$ of the kind displayed on Figures \ref{fig1}, \ref{fig2} and \ref{fig3}. If the structure of repetitions in the field $\Phi$ exhibits some characteristic length $\Lambda$, then on a geodesic segment of length higher (resp. lower) than $\Lambda$, one expects the average number of points where $\Phi$ vanishes to be higher (resp. lower)  than one. This makes the following definition natural.

\begin{defi} \label{typsp} Let $X$ be a Riemannian homogeneous space and let $\Phi$ be a smooth, invariant and centered Gaussian field on $X$. The \emph{typical spacing} of $\Phi$ is the positive number $\Lambda(\Phi)$ such that, for every geodesic segment $\Sigma$ on $X$:
$$ \frac{1}{\Lambda(\Phi)} := \frac{\mathbb{E}\left[ \mathcal{N}_{\Sigma}\right]}{|\Sigma|}. $$ \end{defi}

This does have a meaning: the probability distribution of $\mathcal{N}_{\Sigma}$ depends only on $|\Sigma|$ because of the invariance, and   \eqref{KacRice} shows that that the expectation $\mathbb{E}\left[ \mathcal{N}_{\Sigma}\right]$ depends linearly on~$|\Sigma|$.

For invariant fields with samples in an eigenspace of the Laplacian, a simple application of Rice's formula reveals that the dependence of the typical spacing on the eigenvalue is quite simple:

\begin{prop} \label{Monoch} Let $X$ be a Riemannian homogeneous space, let $\Phi$ be a smooth, centered, invariant Gaussian field on $X$, and let $\beta$ denote the variance of $\Phi(x)$ at any point $x \in X$. 

If the samples of $\Phi$ lie a.s. in the eigenspace $\left\{ f \in \mathcal{C}^\infty(X) \ | \ \Delta_{X} f = E f \right\}$, then the typical spacing $\Lambda(\Phi)$ is given by $$ \Lambda(\Phi) =  \frac{\pi}{{\sqrt{|E|}}}\frac{\sqrt{\dim X}}{\sqrt{\beta}}.$$  \end{prop}

\begin{proof} Let $\gamma$ be a geodesic on $X$, let $x_0$ be a point on $\gamma$ and $\vec{\gamma}$ be the element of the Lie algebra $\mathfrak{g}$ defined at the beginning of \S \ref{sec:spacing}. Let us write $\kappa$ for the second spectral moment of the stationary Gaussian field on the real line, say $\mathbf{u}$, obtained by restricting $\Phi_{\gamma}$ to $\mathbb{R}{\vec{\gamma}}$: the number $\kappa$ is the variance $\mathbb{E}\left[\mathbf{u}'(0)^2\right]$. Because of \eqref{KacRice}, the spacing $\Lambda(\Phi)$ is equal to $\frac{\pi}{\sqrt{\kappa}}$.

 Now, $\mathbf{u}'(0)$ is the derivative of $\Phi$ in the direction $\vec{\gamma}$ (hereafter denoted by $L_{\vec{\gamma}}(\Phi)$). Its variance can be recovered from the second derivative of the covariance function of $\Phi$ in the direction $\vec{\gamma}$. Let us write $\Gamma$ for the covariance function of $\Phi$, turned into a function on $G$ thanks to a choice of base point $x_{0}$ in $X$. Recall that $\Gamma(a^{-1}b) = \mathbb{E}\left[ \Phi(a \cdot x_{0}) \Phi(b \cdot x_{0})\right]$. We can thus evaluate second derivatives in two different ways. If we separately consider the functions $f_{1} : (a,b) \mapsto \Gamma(a^{-1}b)$ and $f_{2} : (a,b) \mapsto  \mathbb{E}\left[ \Phi(a \cdot x_{0}) \Phi(b \cdot x_{0})\right]$ from $G^2$ to $\mathbb{C}$ and write the Lie derivative in the direction $\vec{\gamma}$ with respect to either $a$ or $b$ as $L^{a}_{\vec{\gamma}}$ or $L^{b}_{\vec{\gamma}}$, then we have $\left(L^{a}_{\vec{\gamma}} L^{b}_{\vec{\gamma}} f_{1}\right)(1_{G}) = -L_{\vec{\gamma}}^2\left( \Gamma \right)(1_{G})$, and on the other hand we have $\left(L^{a}_{\vec{\gamma}} L^{b}_{\vec{\gamma}} f_{2}\right)(1_{G}) = \mathbb{E}\left[\left(L_{\vec{\gamma}} \Phi \right)(x_{0})^2 \right]$. The fact that $f_{1} = f_{2}$ then yields
\begin{equation} \label{derivsec} \mathbb{E}\left[\left(L_{\vec{\gamma}} \Phi \right)(x_{0})^2 \right] = -L_{\vec{\gamma}}^2\left( \Gamma \right)(1_{G}).\end{equation}
Recall that we can identify $X$ with the coset space $G/K$, where $K$ is the stabilizer of $x_0$. If $\Gamma_{X}$ is the map $xK \mapsto \Gamma(x)$ from $X=G/K$ to $\mathbb{C}$, then the quantity on the right-hand-side of \eqref{derivsec} is equal to $-L_{\vec{\gamma}}^2\left( \Gamma_{X} \right)(x_{0})$. Of course, the Laplacian on $X$ has much to do with second derivatives: 
\begin{itemize}
\item[$\bullet$] when $X$ is flat and $\Delta_{X}$ is the usual laplacian, we can choose Euclidean coordinates on $X$ such that $\mathbb{R}\vec{\gamma}$ is the first coordinate axis; writing $X_{i}$ for the vector fields generating the translations along the coordinate axes, we then have $\Delta_{X} = \sum \limits_{i=1}^{\dim X} L_{X_{i}}^2$
\item[$\bullet$]In the general case, we can localize the computation and use normal coordinates around $x_{0}$: suppose $(\gamma_{1}^{x_{0}}, ... \gamma_{p}^{x_{0}})$ is an orthonormal basis of $T_{x_{0}} X$, and let $\vec{\gamma}_{1}, ... \vec{\gamma}_{p}$ be elements of $\mathfrak{g}$ whose induced vector fields on $X$ coincide at $x_{0}$ with the $\gamma^{x_0}_{i}$s. Then $\left(\Delta_{X} \Gamma_X\right)(x_{0}) = \sum \limits_{i=1}^{p} \left(L_{\vec{\gamma}_{i}}^2 \Gamma_X\right)(x_{0})$.\end{itemize}

 We now use the fact that the field is $G$-invariant and note that as a consequence, the directional derivatives of $\Gamma_{X}$ at the identity coset are all identical ; so
\begin{equation} \label{switch} \left( L_{\vec{\gamma}}^2 \Gamma_X \right) (x_{0}) =  \frac{1}{\dim X} \left( \Delta_{X} \Gamma_{X} \right)(x_{0}). \end{equation}
 We can now specialize to the case where $\Gamma$ is an eigenfunction of $\Delta_{X}$ for an eigenvalue $E$. Recall that $E$ is nonpositive when the manifold $X$ is compact and nonnegative otherwise; the above then yields
 $$ \kappa = \frac{1}{\dim X} |E| \Gamma(0) =    \frac{1}{\dim X} \beta |E|, $$
 and Proposition \ref{Monoch} follows. \end{proof}

{\begin{rema} The calculation in the proof of Proposition \ref{Monoch} is quite similar to those conducted in \cite[Section 2]{RudnickWigman2D} and \cite[Section 2.2]{RudnickWigman3D} for nodal intersections against a fixed closed curve in the $2$- or $3$-dimensional torus. (I thank the referee for pointing out those papers; note that the calculations in \cite{RudnickWigman2D, RudnickWigman3D} are for any reference curve, not necessarily a geodesic, and the result is independent of the geometry of the curve). The results there, and the methods of proof, do agree with those of Proposition \ref{Monoch}.  \end{rema} }

\begin{exem} \label{eucli} Suppose $X$ is the Euclidean plane, and we start from the monochromatic complex-valued invariant field, say $\mathbf{\Phi}$, with characteristic wavelength $\lambda$. Then its real part $\Phi_{\mathbb{R}}$ has $\beta = 1/2$ and Proposition \ref{Monoch} says that $\Lambda({\Phi}_{\mathbb{R}}) = \lambda$. This we may have expected, since the samples of $\mathbf{\Phi}$ are superpositions of waves with wavelength $\lambda$. \end{exem}

When the curvature is nonzero, however, Proposition \ref{Monoch} seems to say something nontrivial. We shall give two examples of phenomena that it points to: the first is for negatively-curved symmetric spaces, the second is for (positively-curved) compact spaces.

\begin{exem} \label{spacing_hyp} Suppose $X$ is a symmetric space of noncompact type, and we start from a monochromatic invariant field, say ${\Phi}$, with spectral parameter $\omega$ and point-variance $\beta = 1/( \dim X)$. In the notations of \S \ref{subsec:hyperbolique}, we get 
$$ \Lambda(\mathbf{\Phi}) = \frac{2\pi}{\sqrt{|\omega|^2 + |\rho|^2}}.$$
This is not intuitively quite as obvious as Example \ref{eucli} : the samples of $\mathbf{\Phi}$ are superpositions of Helgason waves whose phase surfaces line up at invariant distance $\frac{2\pi}{|\omega|}$. The curvature-induced shift in the typical spacing comes from to the curvature-induced growth factor in the eigenfunctions for $\Delta_{X}$.\end{exem}

\begin{exem} \label{spectralgap} Suppose $X$ is a compact commutative space. Let $c_0$ be the gap between zero and the first nonzero eigenvalue\footnote{When $X$ is a general Riemannian manifold, relating $c_0$ to the geometry of $X$ is a deep question: see for instance \cite{BergerGauduchonMazet}, III.D.} of $\Delta_{X}$. Then $c_0$ provides a nontrivial upper bound for the typical spacing of every smooth and standard invariant field $\Phi$ on $X$: one always has $\Lambda(\Phi) \leq \frac{\pi \dim(X)^{1/2}}{c_0}$. \end{exem}

If $\Phi$ has almost all its samples in an eigenspace of $\Delta_{X}$, the statement in Example \ref{spectralgap} follows from Proposition \ref{Monoch}. If $\Phi$ does not almost surely have its samples in an eigenspace of $\Delta_{X}$, then the statement in Example \ref{spectralgap} can be obtained by splitting $\Phi$ into monochromatic fields using the results of \S \ref{subsec:gelfand}. We should indeed record that if $X$ is commutative and if we start with an \emph{arbitrary} invariant field $\Phi$, we can evaluate the typical spacing of $\Phi$ from that of its monochromatic components:

\begin{lemm} \label{Nonmonoch} Let $X$ be a smooth commutative space, let $\Phi$ be a smooth, invariant, centered, real-valued Gaussian field on $X$, and let $\beta$ denote the variance of $\Phi(x)$ at any point $x \in X$. 

Write the spectral decomposition of $\Phi$ (from \S \ref{subsec:gelfand}) as
$$ \Phi = \int_{\Lambda} \Phi_{\lambda} dP(\lambda); $$
then the typical spacing of $\Phi$ is related to that of the $\Phi_{\lambda}$s as follows:$$ \left(\frac{2\pi}{\Lambda(\Phi)}\right)^2 = {\int_{\Lambda} \left(\frac{2\pi}{\Lambda(\Phi_{\lambda})}\right)^2 dP(\lambda)}.$$
 \end{lemm}

 \begin{proof}
 Let us write $\Gamma$ for the covariance function of $\Phi$, $\varphi_{\lambda}$ for the spherical function with spectral parameter $\lambda$. Note that $\Gamma =   \int_{\Lambda} \varphi_{\lambda} dP(\lambda)$ as we saw, and remember the proof of Proposition \ref{Monoch}: taking up its notations, we there obtained
 $$ \left(\frac{2\pi}{\Lambda(\Phi)}\right)^2 = -L_{\vec{\gamma}}^2\left( \Gamma_{X} \right)(x_{0}).$$
We simply need to evaluate $L_{\vec{\gamma}}^2\left( \Gamma_{X} \right)(x_{0})$. But the relationship with the Laplacian in Eq. \eqref{switch} still holds, and switching the integration with the Lie derivatives yields 
 $$ L_{\vec{\gamma}}^2\left( \Gamma_{X} \right)(x_{0}) = \frac{1}{\dim(X)} \int_{\Lambda}  \Delta_X\left( \varphi_{\lambda} \right)(x_{0}) dP(\lambda) = - {\int_{\Lambda} \left(\frac{2\pi}{\Lambda(\Phi_{\lambda})}\right)^2 dP(\lambda)}$$
 as announced. \end{proof}
 
This does prove the claim about non-monochromatic fields in Example \ref{spectralgap}, and shows more generally that Proposition \ref{Monoch} can in fact yield information about all invariant fields.

\begin{rema} The additional requirement that $X$ be a commutative space in Lemma \ref{Nonmonoch} makes it easy to write down the decomposition, but seems unnecessarily stringent given the proof: at the cost of complicating the notations, one can presumably evaluate the typical spacing of a general field on a Riemannian homogeneous  space $X$ by using spectral theory to split it into fields with samples in an eigenspace of $\Delta_{X}$. \end{rema}


\section{The mean nodal volume for invariant smooth fields on Riemannian homogeneous spaces}\label{sec:densite}

We return to the general situation of a general Riemannian homogeneous space $X$. We shall describe a simple and general relationship between  the average nodal volume of invariant smooth Gaussian fields on $X$, on the one hand, and the typical spacing of \S \ref{sec:spacing}, on the other hand.

The results of \S \ref{sec:spacing} will yield simple and concrete consequences of the relationship when (a) the field takes values in an eigenspace of the Laplacian, so that Proposition \ref{Monoch} can be applied to obtain information about its typical spacing, or when (b) the homogeneous space $X$ is in fact commutative and Lemma \ref{Nonmonoch} can yield information about any invariant field in terms of its spectral decomposition.


\subsection{Statement of the result}

Suppose $\mathbf{\Phi}$ is a centered smooth invariant Gaussian field on $X$ with values in a finite-dimensional vector space $V$. We promised in the Introduction to define a volume unit appropriate to $\mathbf{\Phi}$, but in \S \ref{sec:spacing} provides only a length unit, and applies only to the case where $V$ is one-dimensional. We shall use coordinates on $V$ to define our volume unit.

For each nonzero $u$ in $V$, the typical spacing $\Lambda\left(\langle u | \mathbf{\Phi}\rangle\right)$ of the projection of $\Phi$ on the axis $\mathbb{R} u$ depends on the variance $\beta_{u}$ of the real-valued Gaussian variable $\Lambda\left(\langle u | \mathbf{\Phi}(p)\rangle\right)$ (here $p$ is any point of $X$); however, the quantity $\sqrt{\beta_{u}}  \Lambda\left(\langle u | \mathbf{\Phi}\rangle\right)$ does not depend on $u$.

\begin{defi} Suppose $\mathbf{\Phi}$ is an invariant Gaussian field on $X$ with values in a finite-dimensional vector space $V$. For every orthonormal basis $(u_{1}, ... u_{\dim V})$ of $V$, we can form the quantity $\prod \limits_{i=1}^{\dim V} \sqrt{\beta_{u_i}}  \Lambda\left(\langle u_i | \mathbf{\Phi}\rangle\right)$. This quantity depends only on  $\mathbf{\Phi}$ and not on the chosen basis. We will call it the \emph{volume unit for} $\mathbf{\Phi}$, and write $\mathcal{V}(\mathbf{\Phi})$ for it. \end{defi}

The terminology is perhaps easiest to understand when $\dim V$ and $\dim X$ coincide, provided $\mathbf{\Phi}(p)$ is an isotropic Gaussian vector and $\beta_{u}$ equals $1$ for each $u$. Our definition is inspired by that case, in fact from the case $\dim(V)=\dim(X)=2$, where it corresponds to the notion of \emph{hypercolumn size} from Neuroscience (see \cite{Hubel} for the biological definition, \cite{WolfGeisel} for its geometrical counterpart, and \cite{AAVariance} for comments). 

It seems appropriate to point out that when $\dim V$ and $\dim X$ do not coincide, our $\mathcal{V}(\Phi)$ does not seem to correspond to the volume of any compelling geometrical object. We shall see, however, that it is natural to interpret $\mathcal{V}(\mathbf{\Phi})$ as a volume unit.\\
 
 If $\mathbf{\Phi}$ is a smooth invariant Gaussian field as above, then the zero-set of $\mathbf{\Phi}$ is a random subset of $X$; for all samples of $\mathbf{\Phi}$ for which $0$ is a regular value, that subset is a $(\dim X -\dim V)$-dimensional submanifold (and is empty if $\dim V > \dim X$). Every submanifold of $X$ inherits a metric, and hence a volume form, from that of $X$. For all samples of  $\mathbf{\Phi}$ for which $0$ is a regular value, this gives a meaning to the volume of the intersection of $\mathbf{\Phi}^{-1}(0)$ with a compact subset of $X$. When $A$ is a Borel region of $X$, we wan thus define a real-valued random variable $\mathcal{M}_{{\mathbf{\Phi}, A}}$ by recording the volume of $A \cap \mathbf{\Phi}^{-1}(0)$ for all samples of $\mathbf{\Phi}$ for which $0$ is a regular value, and recording, say, zero for all samples of $\mathbf{\Phi}$ for which $0$ is a singular value.

\begin{theo} \label{Main} Suppose $\mathbf{\Phi}$ is a centered invariant Gaussian random field on a homogeneous space $X$ with values in a Euclidean space $V$. Assume that the individual component processes of the field $\mathbf{\Phi}$ are independent. Write $\mathcal{M}_{{\mathbf{\Phi},A}}$ for the random variable recording the geometric measure of $\mathbf{\Phi}^{-1}(0)$  in a Borel region ${A}$ of $X$, and $\text{Vol}({A})$ for the volume of $A$. Write $\mathcal{V}(\mathbf{\Phi})$ for the volume unit for  $\mathbf{\Phi}$. Then

\begin{equation} \label{theoreme} {\mathbb{E}\left[ \mathcal{M}_{{\mathbf{\Phi}, A}} \right]} \cdot \frac{\mathcal{V}(\mathbf{\Phi})}{\text{\emph{Vol}}({A})} = (\dim V)!  \left(\frac{\pi}{2}\right)^{{(\dim V)}/2}.\end{equation} \end{theo}

\begin{rema} It is the left-hand side of \eqref{theoreme} that makes it natural to interpret $\mathcal{V}(\mathbf{\Phi})$ as a volume unit. It should be put in print that \eqref{theoreme} says that when expressed in the natural volume unit for $\mathbf{\Phi}$, the density of the zero-set for $\mathbf{\Phi}$ depends only on the dimension of the source and target spaces, and \emph{not on the group acting}. \end{rema}

\begin{rema} I am grateful to the referee of an earlier version for pointing out that Theorem \ref{Main} is not far from being an extremely special case of \cite[Theorem 15.9.4]{AdlerTaylor}. (In \cite[Theorem 15.9.4]{AdlerTaylor}, take $i=\dim(X)-\dim(V)$, $M=A$ and $D=\{0\}$; then all but one of the $\mathcal{L}_k$ of \cite{AdlerTaylor}  disappear). There are slightly stronger hypotheses in \cite{AdlerTaylor}, including a compact base manifold and the assumption that  all the $\beta_{u_{i}}$ are $1$. But lifting the assumptions is not difficult (see e.g. \cite{Taylor2} for the case $\beta_{u_i} \neq 1$, and the easy idea of Lemma \ref{Azai} below for the compactness). Our result can be viewed, therefore, as an exposition of what happens in the presence of symmetries.   \end{rema}

\begin{rema} Invariant Gaussian random fields have ergodicity properties (see \cite{Adler}, chapter 6) which make it possible in principle to evaluate the left-hand side of \eqref{theoreme} from a single sample of $\mathbf{\Phi}$. When looking at a single realization of a random field, observing the precise average size for the zero-set expressed by Theorem \ref{Main} can thus be viewed a signature that the field has a symmetry, regardless of the fine structure of the symmetry involved.\end{rema}


\subsection{Proof of Theorem \ref{Main}}

We will use Aza\"is and Wschebor's Kac-Rice formula for random fields  \cite[Theorem 6.8]{AzaisWschebor}. We should clearly state that the proof of Theorem \ref{Main} is a rather direct adaptation of the one which appears for complex-valued fields on the Euclidean plane and space in \cite{AzaisWscheborLeon, AzaisWscheborLeon2}.

 We first recall their formula, adding an immediate adaptation to our situation where the base space is a Riemannian manifold rather than a Euclidean space. Our adaptation is very close to being a particular Gaussian case of Theorem 12.1.1 in \cite{AdlerTaylor}: since the base manifold there is assumed to be compact, however, we include a proof.

\begin{lemm}\label{Azai}

Suppose $(M,g)$ is a Riemannian manifold, $d$ is a positive integer, and $\Phi:M\rightsquigarrow\mathbb{R}^{d}$ is a smooth Gaussian random field. Assume that the variance of the Gaussian vector $\Phi(p)$ at each point $p$ in $M$ is nonzero. For each $p \in \mathbb{R}^d$, let $p_{\Phi(p)}: \mathbb{R}^d \to \mathbb{R}$ denote the density of the Gaussian random vector $\Phi(p)$ of $\mathbb{R}^d$ with respect to Lebesgue measure. For every Borel subset $A$ in $M$, write $\mathcal{M}_{\mathbf{\Phi},A}$ for the random variable recording the geometric measure of $\Phi^{-1}(0)$.

Assume that $0$ is almost surely a regular value of $\Phi$; in other words, assume that 
\begin{equation} \label{nondeg}\mathbb{P} \left\{ \exists p \in M \ , \ \Phi(p) = 0 \enskip \text{and} \enskip d\Phi(p)  \text{ does not have full rank } \right\} = 0.\end{equation}
Then the average size of the zero-set of $\Phi$ is given by
\begin{equation} \label{AzWs} \mathbb{E}\left[\mathcal{M}_{\mathbf{\Phi},A}\right] = \int_{A} \mathbb{E}\left\{ \enskip \det\! \left[d\Phi(p)d\Phi(p)^\dagger \right]^{1/2} \enskip \big| \enskip \Phi(p) = 0\  \right\} \cdot p_{\Phi(p)}(0)\cdot dVol_{g}(p). \end{equation}
\end{lemm}

\begin{proof} 

After splitting $A$ into a suitable family of Borel subsets, we can work in a single volume-preserving chart and assume that $A$ is contained in an open subset $U$ of $M$ for which there exists a volume-preserving diffeomorphism  $\psi : M \supset U \rightarrow \psi(U) \subset \mathbb{R}^{d}$. We turn $\Phi|_{U}$ into a Gaussian random field $\Psi$ on an open subset of $\mathbb{R}^{\dim M}$ by setting
$$ \Psi \circ \psi = \Phi. $$
We can then apply Theorem 6.8 in \cite{AzaisWschebor} to find the mean nodal volume of $\Psi$ in $\psi(A)$; since the nodal volume of $\Psi$ in $\psi(A)$ is that of $\Phi$ in $A$, the theorem yields
$$ \mathbb{E}\left[\mathcal{M}_{\mathbf{\Phi},A}\right] = \int_{\psi(A)} \mathbb{E}\left\{  |\det\left[d\Psi(x)d\Psi(x)^\dagger \right]|^{1/2}  \ | | \Psi(x) = 0 \right\} p_{\Psi(x)}(0)dx, $$
where the volume element is Lebesgue measure.

If we start from the right-hand-side of ~\eqref{AzWs} and change variables using $\psi$, we get \footnotesize
\begin{align*} & \int_{A} \mathbb{E}\left\{  |\det\left[d\Phi(p)d\Phi(p)^\dagger \right]|  \ \big| \Phi(p) = 0 \right\} p_{\Phi(p)}(0)dVol_{g}(p) =  \\
& \int_{\psi(A)}  \mathbb{E}\left\{  |\det\!\left[d\Phi(\psi^{-1} (x))d\Phi(\psi^{-1} (x))^{\dagger}\right]|^{1/2}  \ \big| \Phi(\psi^{-1}(x)) = 0 \right\} p_{\Phi(\psi^{-1}(x))}(0)\left| \det\!\!\left[ d\psi^{-1}(x) \right] \right| dx = \\
& \int_{\psi(A)}  \mathbb{E}\left\{ | \det\!\!\left[ d\psi^{-1}(x) \right] \det\!\!\left[d\Phi(\psi^{-1} x)\right] \det\!\!\left[d\Phi(\psi^{-1} x)^\dagger \right]  \det\!\!\left[ d\psi^{-1}(x)^\dagger \right] |^{1/2}  \big| \Psi(x) = 0 \right\} p_{\Psi(x)}(0)dx =
 \\
& \int_{\psi(A)}  \mathbb{E}\left\{ | \det\!\!\left[ d\Psi(x)  d\Psi(x)^\dagger \right] |^{1/2}  \ \big| \Psi(x) = 0 \right\} p_{\Psi(x)}(0)dx =  \mathbb{E}\left[\mathcal{M}_{\mathbf{\Phi},A}\right] \end{align*}
\normalsize
(between the first line and the second was a change of variables for integration on Riemannian manifolds, between the second line to the third one should simply remember that there is no randomness in $\psi$; going from the third to the fourth one should remember that \[ \det\left[d\Psi(x)\right] = \det\left[d\Phi(\psi^{-1}(x))\right]\det\left[d\psi(x)^{-1}\right]).\] 
\end{proof}

 Let us return to the case where $\mathbf{\Phi}: X \to V$ is an invariant and smooth field on a homogeneous space with independent components. Choose an orthonormal basis $\mathcal{U}=(u_{1}, ... u_{\dim V})$ of $V$. Write $\beta_{i}$ for the standard deviation of the Gaussian variable $\langle u_{i}, \mathbf{\Phi}(p)\rangle$ at each $p$ (because of the invariance, $\beta_i$ does not depend on $p$). Write $\mathcal{V}$ for the quantity $\beta_{1} ... \beta_{\dim V}$, which is the volume of the characteristic ellipsoid for the Gaussian vector $\mathbf{\Phi}(p)$ at each $p$ and depends neither on $p$ nor on the choice of basis in $V$.

 To prove Theorem \ref{Main}  we need to look for for $\mathbb{E}\left[ \mathcal{M}_{{\mathbf{\Phi}, A}} \right]$, and since the field $\mathbf{\Phi}$ is Gaussian, we know that $p_{\Phi(p)}(0) = {(\mathcal{V})^{-1}(2\pi)^{-(\dim V)/2}}$ for each $p$. In addition, because of the invariance we know that $p \mapsto \mathbb{E}\left[ \Phi(p)^2 \right]$ is a constant function on $X$, so for any vector field $\vec{\gamma}$ on $X$, the Lie derivative $L_{\vec{\gamma}} \Phi$ satisfies
\begin{equation} \label{indep} \mathbb{E} \left[(L_{\vec{\gamma}} \Phi)(p) \Phi(p) \right] = 0.\end{equation}
This independence condition is key to our application of Lemma  \ref{Azai}:
\begin{itemize}
\item[$\bullet$] The almost-sure regularity condition in \eqref{nondeg} is a consequence of this independence and of the Bulinskaya-type lemma in \cite[Proposition 6.11]{AzaisWschebor}. For fixed $p$, the independence in \eqref{indep} implies that the pair $(\Phi(p), d\Phi(p))$ is a Gaussian variable in the Euclidean space $V \times \text{End}(T_{p} X, V)$; by choosing local coordinates for $X$ in a neighborhood of $p$ (and remembering that we already chose a basis $\mathcal{U}$ for $V$, we can apply Proposition 6.12 in \cite{AzaisWschebor} to see that the hypothesis  \eqref{nondeg} holds in some neighborhood of $p$; we can then use the second-countability of $X$ to conclude that Lemma \ref{Azai} can actually be applied in a present situation.
 \item[$\bullet$] Moreover, if we choose a basis in $T_{p} X$ (in addition to the chosen basis $\mathcal{U}$ of $V$) and view $d\Phi(p)$ as a matrix, then the entries of that matrix will be Gaussian random variables which are independent from every component of $\Phi(p)$ because of our assumption that all individual component processes are independent. This means we can remove the conditioning in \eqref{AzWs}.
 \end{itemize}
This leads to considerable simplifications in formula \eqref{AzWs}:  
\begin{equation} \label{preuve}\mathbb{E}\left[ \mathcal{M}_{{\mathbf{\Phi}, A}} \right] = \frac{1}{{(2\pi\mathcal{V}^2)}^{(\dim X)/2}} \int_{A}  \mathbb{E}\left\{  |\det\left[d\Phi(p)d\Phi(p)^\dagger \right]|  \right\} dVol_{g}(p). \end{equation}

 Now, $d\Phi(p)$ is a random endomorphism from $T_{p} X$ to $V$. Recall that if $\gamma$ is a tangent vector to $X$ at $p$ and if we observe the probability distribution of the random vector $ (L_{\gamma} \Phi)(p)=d\Phi(p)\cdot \gamma$ in $V$, then the distribution does not depend on $p$, and does not depend on $\gamma$ either. Thus there is a basis $(v_{1}, ... v_{\dim V})$ of $V$ such that for each $\gamma$ in $T_{p}X$, $\langle (L_{\gamma} \Phi)(p), v_{i} \rangle$ is independent from $\langle (L_{\gamma} \Phi)(p), v_{j} \rangle$ if $i \neq j$: the $v_{i}$s generate the principal axes for the Gaussian distribution of $ (L_{\gamma} \Phi)(p)$. If we choose any basis of $T_{p} X$ and write down the corresponding matrix for $d\Phi(p)$ (it has $\dim X$ rows and $\dim V$ columns), then the columns will be independent and will be isotropic Gaussian vectors in $\mathbb{R}^{\dim X}$. We now add the following very elementary (and certainly not very original) remark.

\begin{lemm} \label{parall} Suppose $M$ is a matrix with $n$ rows and $k$ columns, $n \geq k$. Write $(m_{1}, ... m_{k}) \in (\mathbb{R}^n)^k$ for the columns of $M$. Then the determinant of $MM^{\dagger}$ is the square of the volume of the parallelotope $ \left\{ \sum \limits_{i=1}^{k} t_{i} m_{i} \ \big| \ t_{i} \in [0,1] \right\}$. \end{lemm}

\begin{proof} Choose an orthonormal basis $(m_{k+1}, ... m_{n})$ of $\text{Span}(m_{1}, ... m_{k})^\perp$. Then the signed volume of the $k$-dimensional parallelotope $ \left\{ \sum \limits_{i=1}^{k} t_{i} m_{i} \ \big| \ t_{i} \in [0,1] \right\}$ is the same as that of the $n$-dimensional parallelotope $ \left\{ \sum \limits_{i=1}^{n} t_{i} m_{i} \ \big| \ t_{i} \in [0,1] \right\}$. Now, write $\tilde{M}$ for the $n\times n$ matrix whose columns are the coordinates of the $m_{i}$ in the canonical basis of $\mathbb{R}^n$. Then $\tilde{M}\tilde{M}^{\dagger}$ is block-diagonal, one block is $MM^{\dagger}$ and the other block is the identity because $(m_{k+1}, ... m_{n})$ is an orthonormal family.  Thus the determinant of $MM^\dagger$ is the square of that of $\tilde{M}$, and $\det(\tilde{M})$ is the volume of the parallelotope $\left\{ \sum \limits_{i=1}^{n} \alpha_{i} m_{i} \ \big| \ \alpha_{i} \in [0,1] \right\}$. \end{proof}

 Combining Lemma \ref{parall} and Equation \eqref{preuve}, we are led to calculating the mean Hausdorff volume of the random parallelotope generated by $\dim(V)$ independent isotropic Gaussian vectors in $\mathbb{R}^{\dim X}$.

\begin{lemm} \label{pave} Suppose $u_{1}, ... u_{k}$ are independent isotropic Gaussian vectors with values in $\mathbb{R}^n$, so that the probability distribution of $u_{i}$ is $x \mapsto \frac{1}{\alpha_{i} \sqrt{2\pi}} e^{-\norme{x}^2/2\alpha_{i}^2}$. Write $\mathcal{V}$ for the characteristic volume $\alpha_{1}... \alpha_{k}$, and write $\mathbf{V}$ for the random variable recording the $k$-dimensional volume of the parallelotope $\left\{ \sum \limits_{i=1}^{k} t_{i} u_{i} \ \big| \ t_{i} \in [0,1] \right\}$. Then$$ \mathbb{E}[\mathbf{V}] = (k!) \mathcal{V}. $$
\end{lemm}

\begin{proof} Let us start with $k$ (deterministic) vectors in $\mathbb{R}^n$, say $u_{1}^0, ... u_{k}^0$, and choose a basis $u_{k+1}^0, ... u_{n}^0$ for $\text{Span}(u_{1}^0, ... u_{n}^0)^\perp$. Since $ \text{det} (u_{1}^0, .... u_{k}0) = \text{det} (u_{1}^0, .... u_{n}0) $ is the (signed) volume of the parallelotope generated by the $u_{i}^0$s, we can use the ``base times height'' formula: writing $P_{V}$ for the orthogonal projection from $\mathbb{R}^n$ onto a subspace $V$,
$$ \text{Vol} (u_{1}^0, .... u_{n}^0) = \norme{P_{\text{Span}(u_{2}^0, ... u_{n}^0)^\perp}(u_{1}^0)} \text{Vol}\left(u_{2}^0, ... u_{n}^0\right). $$  
Of course then 
$$  \text{Vol} (u_{1}^0, .... u_{n}^0) = \prod \limits_{i=1}^{k} \norme{P_{\text{Span}(u_{i+1}^0, ... u_{n}^0)^\perp}(u_{i}^0)}. $$
We now return to the situation with random vectors. Because $u_{1}, ... u_{k}$ are independent, the above formula becomes
$$\mathbb{E}\left[ \text{Vol} (u_{1}, .... u_{k}) \right] = \prod\limits_{i=1}^{k}  \mathbb{E}\left[ N(u_{i}, V^{i} ) \right] $$
where $N(u_{i}, V^{i} )$ is the random variable recording the norm of the projection of $u_{i}$ on any $(i)$-dimensinal subspace of $\mathbb{R}^n$. The projection is a Gaussian vector, and so its norm has a chi-squared distribution with $i$ degrees of freedom. Given the probability distribution of $u_{i}$, the expectation for the norm is then $i \alpha_{i}$, and this  proves Lemma \ref{pave}. \end{proof}

We note that the argument in the above two lemmas is very close to that in \cite{Kabluchko}.

 To complete the proof of Theorem \ref{Main}, choose an orthonormal basis $(\gamma_{1}, ... \gamma_{n})$ in $T_{p} X$. Apply Lemma \ref{pave} to the family $\left(\left(L_{\gamma_{i}} \langle v_{i}, \mathbf{\Phi} \rangle\right)(p)\right)_{i=1..n}$. Then ~\eqref{preuve} becomes
$$ \mathbb{E}\left[ \mathcal{M}_{{\mathbf{\Phi}, A}} \right] = \frac{1}{{(2\pi)}^{(\dim V)/2}\mathcal{V}} \text{Vol}(A) \cdot (\dim V)! \cdot \prod \limits_{i=1}^{\dim V} \mathbb{E}\left[\left(L_{\gamma_{1}} (\langle v_{i}, \mathbf{\Phi} \rangle)(x_{0})\right)^2 \right]^{1/2}. $$
 To bring the typical spacing back into the picture, recall that the definition and the Kac-Rice formula \eqref{KacRice} say that $\mathbb{E}\left[\left(L_{\gamma_{1}} (\langle v_{i}, \mathbf{\Phi} \rangle)(x_{0})\right)^2 \right]^{1/2}$ is none other than $\frac{\pi}{\Lambda(\langle v_{i}, \mathbf{\Phi}\rangle)}$. Thus 
 $$  \mathcal{M}_{{\mathbf{\Phi}, A}}  \frac{\mathcal{V} \prod \limits_{i=1}^{d} \Lambda(\langle v_{i}, \mathbf{\Phi}\rangle)}{\text{Vol}(A)}= \frac{\pi^{\dim V}(\dim V)!}{{(2\pi)}^{(\dim V)/2}} . $$
Since $\mathcal{V} \prod \limits_{i=1}^{d}\Lambda(\langle v_{i}, \mathbf{\Phi}\rangle) $ is by definition the volume unit for $\mathbf{\Phi}$, Theorem \ref{Main} is established. 

\qed

\bibliographystyle{smfplain}
\bibliography{afgoustidis_nodal_volume_references}

\providecommand{\bysame}{\leavevmode ---\ }
\providecommand{\og}{``}
\providecommand{\fg}{''}
\providecommand{\smfandname}{\&}
\providecommand{\smfedsname}{\'eds.}
\providecommand{\smfedname}{\'ed.}
\providecommand{\smfmastersthesisname}{M\'emoire}
\providecommand{\smfphdthesisname}{Th\`ese}
\begin{thebibliography}{10}

\bibitem{AbertBergeronLeMasson}
{\scshape M.~Abert, N.~Bergeron {\normalfont \smfandname} E.~Le~Masson} -- {\og
  Eigenfunctions and random waves in the {B}enjamini--{S}chramm limit\fg},
  Preprint,
  \href{https://doi.org/10.48550/arXiv.1810.05601v2}{arXiv:1810.05601v2}.

\bibitem{Adler}
{\scshape R.~J. Adler} -- \emph{The geometry of random fields}, John Wiley \&
  Sons, Ltd., Chichester, 1981, Wiley Series in Probability and Mathematical
  Statistics.

\bibitem{AdlerTaylor}
{\scshape R.~J. Adler {\normalfont \smfandname} J.~E. Taylor} -- \emph{Random
  fields and geometry}, Springer Monographs in Mathematics, Springer, New York,
  2007.

\bibitem{AAVariance}
{\scshape A.~{Afgoustidis}} -- {\og Monochromaticity of orientation maps in
  {V}1 implies minimum variance for hypercolumn size\fg}, \emph{J. Math.
  Neurosci.} \textbf{5} (2015), p.~Art. 10, 19.

\bibitem{AAPinwheels}
\bysame , {\og Orientation maps in {V}1 and non-{E}uclidean geometry\fg},
  \emph{J. Math. Neurosci.} \textbf{5} (2015), p.~Art. 12, 45.

\bibitem{AAThese}
{\scshape A.~Afgoustidis} -- {\og Repr\'esentations de groupes de lie et
  fonctionnement g\'eom\'etrique du cerveau\fg}, \smfphdthesisname,
  Universit\'e Paris-7, July 2016.

\bibitem{Anantharaman}
{\scshape N.~Anantharaman} -- {\og Topologie des hypersurfaces nodales de
  fonctions al\'{e}atoires gaussiennes (d'apr\`es {N}azarov et {S}odin, {G}ayet
  et {W}elschinger)\fg}, \emph{S\'{e}minaire Bourbaki, Expos\'{e} 1116,
  \emph{Ast\'{e}risque \textbf{390}, p. 369-408}} (2017).

\bibitem{AzaisWscheborLeon2}
{\scshape J.-M. Aza\"{i}s, J.~R. Le\'{o}n {\normalfont \smfandname}
  M.~Wschebor} -- {\og Rice formulae and {G}aussian waves\fg}, \emph{Bernoulli}
  \textbf{17} (2011), no.~1, p.~170--193.

\bibitem{AzaisWscheborLeon}
\bysame , {\og Rice formulas and {G}aussian waves {II}\fg}, \emph{Publ. Mat.
  Urug.} \textbf{12} (2011), p.~15--38.

\bibitem{AzaisWschebor}
{\scshape J.-M. Aza\"{i}s {\normalfont \smfandname} M.~Wschebor} -- \emph{Level
  sets and extrema of random processes and fields}, John Wiley \& Sons, Inc.,
  Hoboken, NJ, 2009.

\bibitem{BaldiMarinucci}
{\scshape P.~Baldi, D.~Marinucci {\normalfont \smfandname} V.~S. Varadarajan}
  -- {\og On the characterization of isotropic {G}aussian fields on homogeneous
  spaces of compact groups\fg}, \emph{Electron. Comm. Probab.} \textbf{12}
  (2007), p.~291--302.

\bibitem{Belyaev}
{\scshape Y.~K. Belyaev} -- {\og Continuity and {H}\"{o}lder's conditions for
  sample functions of stationary {G}aussian processes\fg}, in \emph{Proc. 4th
  {B}erkeley {S}ympos. {M}ath. {S}tatist. and {P}rob., {V}ol. {II}}, Univ.
  California Press, Berkeley, Calif., 1961, p.~23--33.

\bibitem{Berard}
{\scshape P.~B\'{e}rard} -- {\og Volume des ensembles nodaux des fonctions
  propres du laplacien\fg}, in \emph{S\'{e}minaire de {T}h\'{e}orie {S}pectrale
  et {G}\'{e}om\'{e}trie, {A}nn\'{e}e 1984--1985}, Univ. Grenoble I,
  Saint-Martin-d'H\`eres, 1985, p.~IV.1--IV.9.

\bibitem{BergerGauduchonMazet}
{\scshape M.~Berger, P.~Gauduchon {\normalfont \smfandname} E.~Mazet} --
  \emph{Le spectre d'une vari\'{e}t\'{e} riemannienne}, Lecture Notes in
  Mathematics, Vol. 194, Springer-Verlag, Berlin-New York, 1971.

\bibitem{BerryDennis}
{\scshape M.~V. Berry {\normalfont \smfandname} M.~R. Dennis} -- {\og Phase
  singularities in isotropic random waves\fg}, \emph{R. Soc. Lond. Proc. Ser. A
  Math. Phys. Eng. Sci.} \textbf{456} (2000), no.~2001, p.~2059--2079.

\bibitem{Cammarota}
{\scshape V.~Cammarota} -- {\og Nodal area distribution for arithmetic random
  waves\fg}, \emph{Trans. Amer. Math. Soc.} \textbf{372} (2019), no.~5,
  p.~3539--3564.

\bibitem{PeccaReel}
{\scshape F.~Dalmao, I.~Nourdin, G.~Peccati {\normalfont \smfandname} M.~Rossi}
  -- {\og Phase singularities in complex arithmetic random waves\fg},
  \emph{Electron. J. Probab.} \textbf{24} (2019), p.~Paper No. 71, 45.

\bibitem{VanDijk}
{\scshape G.~van Dijk} -- \emph{Introduction to harmonic analysis and
  generalized {G}elfand pairs}, De Gruyter Studies in Mathematics, vol.~36,
  Walter de Gruyter \& Co., Berlin, 2009.

\bibitem{DonnellyFefferman}
{\scshape H.~Donnelly {\normalfont \smfandname} C.~Fefferman} -- {\og Nodal
  sets of eigenfunctions on {R}iemannian manifolds\fg}, \emph{Invent. Math.}
  \textbf{93} (1988), no.~1, p.~161--183.

\bibitem{Doob}
{\scshape J.~L. Doob} -- \emph{Stochastic processes}, Wiley Classics Library,
  John Wiley \& Sons, Inc., New York, 1990, Reprint of the 1953 original, A
  Wiley-Interscience Publication.

\bibitem{Faraut}
{\scshape J.~Faraut} -- {\og Finite and infinite dimensional spherical
  analysis\fg}, in \emph{Summer School: large N limits, Bitche}, 2008.

\bibitem{GalerneMorel}
{\scshape B.~Galerne, Y.~Gousseau {\normalfont \smfandname} J.-M. Morel} --
  {\og Random phase textures: theory and synthesis\fg}, \emph{IEEE Trans. Image
  Process.} \textbf{20} (2011), no.~1, p.~257--267.

\bibitem{HelgasonGASS}
{\scshape S.~Helgason} -- \emph{Geometric analysis on symmetric spaces}, second
  \smfedname, Mathematical Surveys and Monographs, vol.~39, American
  Mathematical Society, Providence, RI, 2008.

\bibitem{Hida}
{\scshape T.~Hida {\normalfont \smfandname} M.~Hitsuda} -- \emph{Gaussian
  processes}, Translations of Mathematical Monographs, vol. 120, American
  Mathematical Society, Providence, RI, 1993, Translated from the 1976 Japanese
  original by the authors.

\bibitem{Hubel}
{\scshape D.~H. Hubel} -- \emph{Eye, brain, and vision.}, Scientific American
  Library/Scientific American Books, 1995.

\bibitem{Kaschube}
{\scshape M.~Kaschube, M.~Schnabel, S.~L{\"o}wel, D.~M. Coppola, L.~E. White
  {\normalfont \smfandname} F.~Wolf} -- {\og Universality in the evolution of
  orientation columns in the visual cortex\fg}, \emph{{S}cience} (2010),
  p.~1194869.

\bibitem{KnappPrincetonBook}
{\scshape A.~W. Knapp} -- \emph{Representation theory of semisimple groups},
  Princeton Landmarks in Mathematics, Princeton University Press, Princeton,
  NJ, 2001, Reprint of the 1986 original.

\bibitem{KnappLieGroups}
\bysame , \emph{Lie groups beyond an introduction}, second \smfedname, Progress
  in Mathematics, vol. 140, Birkh\"auser Boston, Inc., Boston, MA, 2002.

\bibitem{Kolmogorov}
{\scshape A.~Kolmogorov} -- {\og Problems of probability theory\fg},
  \emph{Theory of Probability \& Its Applications} \textbf{38} (1994), no.~2,
  p.~177--178.

\bibitem{KKW}
{\scshape M.~Krishnapur, P.~Kurlberg {\normalfont \smfandname} I.~Wigman} --
  {\og Nodal length fluctuations for arithmetic random waves\fg}, \emph{Ann. of
  Math. (2)} \textbf{177} (2013), no.~2, p.~699--737.

\bibitem{Logunov}
{\scshape A.~Logunov} -- {\og Nodal sets of {L}aplace eigenfunctions: proof of
  {N}adirashvili's conjecture and of the lower bound in {Y}au's conjecture\fg},
  \emph{Ann. of Math. (2)} \textbf{187} (2018), no.~1, p.~241--262.

\bibitem{Malyarenko}
{\scshape A.~Malyarenko} -- \emph{Invariant random fields on spaces with a
  group action}, Probability and its Applications (New York), Springer,
  Heidelberg, 2013, With a foreword by Nikolai Leonenko.

\bibitem{MarinucciPeccatti}
{\scshape D.~Marinucci {\normalfont \smfandname} G.~Peccati} -- \emph{Random
  fields on the sphere}, London Mathematical Society Lecture Note Series, vol.
  389, Cambridge University Press, Cambridge, 2011, Representation, limit
  theorems and cosmological applications.

\bibitem{Pecca}
{\scshape D.~Marinucci, G.~Peccati, M.~Rossi {\normalfont \smfandname}
  I.~Wigman} -- {\og Non-universality of nodal length distribution for
  arithmetic random waves\fg}, \emph{Geom. Funct. Anal.} \textbf{26} (2016),
  no.~3, p.~926--960.

\bibitem{Rossi}
{\scshape D.~Marinucci {\normalfont \smfandname} M.~Rossi} -- {\og
  Stein-{M}alliavin approximations for nonlinear functionals of random
  eigenfunctions on {$\Bbb{S}^d$}\fg}, \emph{J. Funct. Anal.} \textbf{268}
  (2015), no.~8, p.~2379--2420.

\bibitem{Taylor2}
{\scshape S.~Panigrahi, J.~Taylor {\normalfont \smfandname} S.~Vadlamani} --
  {\og Kinematic formula for heterogeneous {G}aussian related fields\fg},
  \emph{Stochastic Process. Appl.} \textbf{129} (2019), no.~7, p.~2437--2465.

\bibitem{RossiWigman}
{\scshape M.~Rossi {\normalfont \smfandname} I.~Wigman} -- {\og Asymptotic
  distribution of nodal intersections for arithmetic random waves\fg},
  \emph{Nonlinearity} \textbf{31} (2018), no.~10, p.~4472--4516.

\bibitem{RudnickWigman}
{\scshape Z.~Rudnick {\normalfont \smfandname} I.~Wigman} -- {\og On the volume
  of nodal sets for eigenfunctions of the {L}aplacian on the torus\fg},
  \emph{Ann. Henri Poincar\'{e}} \textbf{9} (2008), no.~1, p.~109--130.

\bibitem{RudnickWigman2D}
\bysame , {\og Nodal intersections for random eigenfunctions on the torus\fg},
  \emph{Amer. J. Math.} \textbf{138} (2016), no.~6, p.~1605--1644.

\bibitem{RudnickWigman3D}
{\scshape Z.~Rudnick, I.~Wigman {\normalfont \smfandname} N.~Yesha} -- {\og
  Nodal intersections for random waves on the 3-dimensional torus\fg},
  \emph{Ann. Inst. Fourier (Grenoble)} \textbf{66} (2016), no.~6,
  p.~2455--2484.

\bibitem{Vilenkin}
{\scshape N.~J. Vilenkin {\normalfont \smfandname} A.~U. Klimyk} --
  \emph{Representation of {L}ie groups and special functions. {V}ol. 1},
  Mathematics and its Applications (Soviet Series), vol.~72, Kluwer Academic
  Publishers Group, Dordrecht, 1991, Simplest Lie groups, special functions and
  integral transforms, Translated from the Russian by V. A. Groza and A. A.
  Groza.

\bibitem{Wei}
{\scshape L.-Y. Wei, S.~Lefebvre, V.~Kwatra {\normalfont \smfandname} G.~Turk}
  -- {\og State of the art in example-based texture synthesis\fg}, in
  \emph{Eurographics 2009, State of the Art Report, EG-STAR}, Eurographics
  Association, 2009, p.~93--117.

\bibitem{WG2003}
{\scshape F.~Wolf {\normalfont \smfandname} T.~Geisel} -- {\og Universality in
  visual cortical pattern formation\fg}, \emph{Journal of Physiology-Paris}
  \textbf{97} (2003), no.~2-3, p.~253--264.

\bibitem{WolfGeisel}
{\scshape F.~Wolf {\normalfont \smfandname} T.~Geisel} -- {\og Spontaneous
  pinwheel annihilation during visual development\fg}, \emph{Nature}
  \textbf{395} (1998), no.~6697, p.~73.

\bibitem{Wolf}
{\scshape J.~A. Wolf} -- \emph{Harmonic analysis on commutative spaces},
  Mathematical Surveys and Monographs, vol. 142, American Mathematical Society,
  Providence, RI, 2007.

\bibitem{WolfSCC}
\bysame , \emph{Spaces of constant curvature}, sixth \smfedname, AMS Chelsea
  Publishing, Providence, RI, 2011.

\bibitem{Yaglom}
{\scshape A.~M. Yaglom} -- {\og Second-order homogeneous random fields\fg}, in
  \emph{Proc. 4th {B}erkeley {S}ympos. {M}ath. {S}tatist. and {P}rob., {V}ol.
  {II}}, Univ. California Press, Berkeley, Calif., 1961, p.~593--622.

\bibitem{Yau1}
{\scshape S.~T. Yau} -- {\og Survey on partial differential equations in
  differential geometry\fg}, in \emph{Seminar on {D}ifferential {G}eometry},
  Ann. of Math. Stud., vol. 102, Princeton Univ. Press, Princeton, N.J., 1982,
  p.~3--71.

\bibitem{Yau2}
{\scshape S.-T. Yau} -- {\og Open problems in geometry\fg}, in
  \emph{Differential geometry: partial differential equations on manifolds
  ({L}os {A}ngeles, {CA}, 1990)}, Proc. Sympos. Pure Math., vol.~54, Amer.
  Math. Soc., Providence, RI, 1993, p.~1--28.

\bibitem{Kabluchko}
{\scshape D.~N. Zaporozhets {\normalfont \smfandname} Z.~Kabluchko} -- {\og
  Random determinants, mixed volumes of ellipsoids, and zeros of {G}aussian
  random fields\fg}, \emph{Zap. Nauchn. Sem. S.-Peterburg. Otdel. Mat. Inst.
  Steklov. (POMI)} \textbf{408} (2012), no.~Veroyatnost i Statistika. 18,
  p.~187--196, 327.

\bibitem{Zerva}
{\scshape A.~Zerva} -- \emph{Spatial variation of seismic ground motions:
  modeling and engineering applications}, Crc Press, 2016.

\end{thebibliography}

\end{document}